\documentclass{amsart}
\usepackage{lmodern}
\usepackage[english]{babel}
\usepackage{amsmath}
\usepackage{amssymb}
\usepackage{mathptmx} 
\usepackage{color}
\usepackage{graphicx}
\usepackage{caption}
\usepackage{subcaption}
\usepackage{float}
\usepackage[all,cmtip]{xy}
\usepackage{bm}
\usepackage{algorithm}
\usepackage{enumitem}
\newlist{steps}{enumerate}{1}
\setlist[steps, 1]{label = Step \arabic*:}
\usepackage{pinlabel}
\newtheorem{theorem}{\bf Theorem}[section]
\newtheorem{lemma}[theorem]{\bf Lemma}

\newtheorem{corollary}[theorem]{\bf Corollary}

\newcommand{\rme}{\mathrm{e}}
\newcommand{\rmi}{\mathrm{i}}

\newcommand{\sign}{\operatorname{sign}}

\newcommand{\defeq}{\mathrel{\mathop:}=}

\begin{document}

\title{All links are semiholomorphic}

\author{
Benjamin Bode}
\date{}

\address{Instituto de Ciencias Matemáticas, CSIC, 28049 Madrid, Spain}
\email{benjamin.bode@icmat.es}





%

\maketitle
\begin{abstract}
Semiholomorphic polynomials are functions $f:\mathbb{C}^2\to\mathbb{C}$ that can be written as polynomials in complex variables $u$, $v$ and the complex conjugate $\overline{v}$. We prove the semiholomorphic analogoue of Akbulut's and King's \emph{All knots are algebraic}, that is, every link type in the 3-sphere arises as the link of a weakly isolated singularity of a semiholomorphic polynomial. Our proof is constructive, which allows us to obtain an upper bound on the polynomial degree of the constructed functions.
\end{abstract}
\section{Introduction}
In \cite{ak} Akbulut and King proved that ``All knots are algebraic''. Since the word \emph{algebraic} carries several different meanings, their title could cause confusion. Besides links that are built from rational tangles as studied by Conway \cite{conway} the term ``algebraic link'' is nowadays usually reserved for links of isolated singularities of complex hypersurfaces. These are known to be unions of certain iterated cables of torus links. So clearly not all knots are algebraic in this sense. If the word ``algebraic'' is interpreted as ``an algebraic set'', then Akbulut's and King's title is a true statement, but also a classical and well-known result in algebraic geometry, see for example the Nash-Tognoli theorem \cite{bcr:1998real}.

The correct interpretation of Akbulut's and King's ``algebraic knots'' lies somewhere between the notion of the link of an isolated singularity and an algebraic set. Consider a real polynomial map $f:\mathbb{R}^4\to\mathbb{R}^2$. A critical point of $f$ is a point $p\in\mathbb{R}^4$, where the real Jacobian matrix $Df(p)$ of $f$ does not have full rank. We say that the origin $O\in\mathbb{R}^4$ is a \textit{weakly isolated singularity} of $f$, if $f(O)=0$, $Df(O)=0$ (i.e., the $2\times4$-matrix with zero entries) and there is some neighbourhood $U$ of the origin such that $U\backslash\{O\}\cap f^{-1}(0)$ contains no critical points of $f$. Hence $f$ is allowed to have a line of critical points pass through the origin, but the origin should be an isolated intersection point of $f^{-1}(0)$ and the critical set.

Every weakly isolated singularity can be associated with a link, since the link type of the intersection of $f^{-1}(0)$ and the 3-sphere $S_{\rho}^3$ of radius $\rho$ does not depend on the sufficiently small radius $\rho>0$. We call $L_f=f^{-1}(0)\cap S_{\rho}^3$ the link of the singularity.

Akbulut and King prove that every link in the 3-sphere arises as the link of a weakly isolated singularity of a real polynomial map $f:\mathbb{R}^4\to\mathbb{R}^2$. Thus their interpretation of the term ``algebraic'' does involve singularities, but of real polynomial maps instead of complex ones, and their definition of an isolated singularity is (as the name suggests) so weak that there is no restriction in the type of links that can be obtained this way. 

Note that by composing an inverse stereographic projection $\mathbb{R}^3\to S^3_{\rho}$ with $f$ and clearing the denominator we obtain a real polynomial map on $\mathbb{R}^3$ whose variety is isotopic to the link of the singularity $L_f$ of $f$, so that Akbulut's and King's proof also establishes $L_f$ as an algebraic set in $\mathbb{R}^3$.


In \cite{bode:ralg}, we discuss a construction of weakly isolated singularities for certain links. It produces functions $f:\mathbb{C}^2\to\mathbb{C}$ that can be written as polynomials in complex variables $u$, $v$ and the complex conjugate $\bar{v}$. Hence they are holomorphic with respect to one complex variable, but not necessarily with respect to the other. We call such functions \textit{semiholomorphic}. They form an interesting family of mixed polynomials \cite{oka}, lying between the complex and the real setting. However, the construction in \cite{bode:ralg} only works for links that satisfy certain symmetry constraints, such as being the closure of a 2-periodic braid. This is necessary in order to obtain polynomials with the desired properties rather than more general real analytic maps.

In this paper we offer a constructive proof of Akbulut's and King's result, an algorithm that takes a braid word as input and produces a polynomial with a weakly isolated singularity, whose link is ambient isotopic to the closure of the given braid. Furthermore, all of the constructed polynomials are semiholomorphic.

\begin{theorem}\label{thm:ak}
Algorithm 1 (outlined in Section 3) constructs for any given braid $B$ on $s$ strands a semiholomorphic polynomial $f:\mathbb{C}^2\to\mathbb{C}$ with $\deg_u(f)=s$, with a weakly isolated singularity at the origin and $L_f$ ambient isotopic to the closure of $B$.
\end{theorem}

The algorithm is based on trigonometric interpolation, which allows us to prove upper bounds on the polynomial degrees of the constructed functions.

\begin{theorem}\label{thm:bound}
Let $B$ be a braid with $s$ strands, $\ell$ crossings and let $\mathcal{C}$ denote the set of components of its closure, which by assumption is not the unknot. Let $s_C$ denote the number of strands of the component $C\in\mathcal{C}$. Then the degree of the polynomial $f$ that Algorithm 1 constructs from the input $B$ is at most:
\begin{equation}
\deg(f)\leq s\ell(2+s)+1+\underset{C\in\mathcal{C}}{\sum}s_C^2\ell.
\end{equation}
\end{theorem}

\begin{corollary}\label{cor}
If the closure of $B$ is a non-trivial knot, the degree of the polynomial $f$ constructed by Algorithm 1 is bounded by
\begin{equation}
\deg(f)\leq 2s\ell(s+1)+1.
\end{equation}
\end{corollary}

We would also like to point out that there is a stronger notion of isolation of singularities of real polynomial maps. We say that the origin is an \emph{isolated singularity} if $f(O)=0$, $Df(O)=0$ and $U\backslash\{O\}$ contains no critical points of $f$. Typically the set of critical points of $f$ is 1-dimensional, so that polynomials with isolated singularities are very rare. The links that arise from isolated singularities, the \emph{real algebraic links}, have not been classified yet and are conjectured to be equal to the set of fibered links \cite{benedetti}. Some constructions of isolated singularities have been put forward to make progress on this conjecture \cite{bode:ralg, looijenga, perron}, but the family of links that are known to be real algebraic is still comparatively small. A construction similar to Algorithm 1, which produces isolated singularities for a large families of fibered links will be subject of a future paper.

Our algorithm can be interpreted as a deformation of a Newton degenerate mixed function in the sense of \cite{oka} or \cite{eder}. Our results can thus be viewed in the broader context of the question: How do deformations of real polynomial mappings affect the topology of their zeros close to singular points? Some work has been done on this question regarding so-called inner  Newton non-degenerate mixed functions \cite{eder} and complex polynomial mappings \cite{fukui, king, saeki}, but the problem is still wide open in the general setting.

The remainder of this paper is structured as follows. Section \ref{sec2} reviews some useful background and introduces notation and conventions. In Section \ref{sec:algo} we give an overview of the algorithm that constructs weakly isolated singularities for any given link. The individual steps of the algorithm are discussed in Section \ref{sec:steps}, where we illustrate that all the steps can indeed be performed algorithmically. We prove our main result Theorem \ref{thm:ak} in Section \ref{sec:proof} by proving that the described algorithm constructs weakly isolated singularities for any given link. The bounds on the polynomial degrees are shown in Section \ref{sec:bounds}.
\ \\
\textbf{Acknowledgments:} The author is grateful to Raimundo Nonato Araújo dos Santos and Eder Leandro Sanchez Quiceno for discussions and feedback on the paper. This work is supported by the European Union’s Horizon 2020 Research and Innovation Programme under the Marie Sklodowska-Curie grant agreement No 101023017.

\section{Background}\label{sec2}

Semiholomorphic polynomials are a special type of \textit{mixed polynomials} as introduced by Oka \cite{oka}. In the dimensions that we are interested in, the set of mixed polynomials $f:\mathbb{C}^2\to\mathbb{C}$ consists of polynomials in two complex variables $u$ and $v$, and their complex conjugates, $\overline{u}$ and $\overline{v}$, so that $f$ takes the form
\begin{equation}
f(u,v)=\underset{i,j,k,\ell}{\sum}c_{i,j,k,\ell}u^i\overline{u}^jv^k\overline{v}^\ell,
\end{equation}
with all but finitely many $c_{i,j,k,\ell}\in\mathbb{C}$ equal to zero. Note that every polynomial map from $\mathbb{R}^4$ to $\mathbb{R}^2$ can be written as a mixed polynomial. A mixed polynomial is semiholomorphic if and only if $c_{i,j,k,\ell}\neq0$ implies $j=0$. Thus a semiholomorphic polynomial is holomorphic with respect to the variable $u$, but not necessarily with respect to the variable $v$.

Semiholomorphic polynomials lend themselves to constructions like the one discussed in this paper for two reasons. First, the holomorphicity in one variable grants us a certain rigidity and control over the behaviour of zeros that is usually associated with complex functions. For instance, we know that for any fixed value of $v=v_*$ the number of zeros of $f(\cdot,v_*)$ is equal to its degree. The second advantage of working with semiholomorphic polynomials is that it is comparatively easy to prove that a point is a regular point, i.e., that the real Jacobian matrix has full rank, and consequently to prove that a singularity is weakly isolated. It suffices to show that the origin $O$ is the only zero of $f$ where $\tfrac{\partial f}{\partial u}$ vanishes.

As for complex polynomials there is the notion of a Newton polyhedron for mixed polynomials \cite{oka}. For every weight vector $P=(p_1,p_2)\in\mathbb{N}^2$ we can define the radially weighted degree of a mixed monomial $M=c_{i,j,k,\ell}u^i\overline{u}^jv^k\overline{v}^\ell$ with respect to $P$ as $d(P;M):=p_1(i+j)+p_2(k+\ell)$. A mixed polynomial $f$ is \textit{radially weighted homogeneous} of degree $d(P;f)$ if there is a weight vector $P$ such that all non-zero monomials $M$ in $f$ have the same radially weighted degree $d(P;M)=d(P;f)$ with respect to $P$.

Our algorithm is based on braids and the fact that every link is the closure of some braid \cite{alexander:1923lemma}. A braid on $s$ strands is a collection of $s$ disjoint curves $(u_j(t),t)\subset\mathbb{C}\times [0,2\pi]$, $j=1,2,\ldots,s$, parametrized by their height coordinate $t$ going from $0$ to $2\pi$. The functions $u_j:[0,2\pi]\to\mathbb{C}$ are assumed to be smooth and to satisfy that for every $j\in\{1,2,\ldots,s\}$ there is a unique $i\in\{1,2,\ldots,s\}$ such that $u_j(2\pi)=u_i(0)$.

Identifying the $t=0$- and the $t=2\pi$-plane results in a link in $\mathbb{C}\times S^1$, the closed braid. Embedding the open solid torus $\mathbb{C}\times S^1$ as an untwisted neighbourhood of the unknot in the 3-sphere $S^3$ defines a link in $S^3$, the closure of the braid, whose link type is well-defined.

Projecting curves via the map $(u,t)\mapsto(\text{Re}(u),t)$ into $\mathbb{R}\times[0,2\pi]$ results in $s$ intersecting curves. A \textit{braid diagram} is such a projection where every intersection is transverse and involves exactly two strands. We keep track of the information about the $\text{Im}(u)$-coordinate of these two strands at the crossing by deleting the strand with the larger $\text{Im}(u)$-coordinate in a neighbourhood of the crossing. The strand with the smaller $\text{Im}(u)$-coordinate is thus the overcrossing strand. This is an arbitrary choice and in previous papers we have not been consistent with our choices (although consistent in each individual paper). Changing this convention only means that several signs throughout this paper need to be reversed. Obviously, the results are not affected by this.

If two strands cross at $t=t_k$ and for all small $\varepsilon>0$ the overcrossing strand has smaller $\text{Re}(u)$-coordinate for all $t\in(t_k-\varepsilon,t_k)$, this crossing is a positive crossing. Non-positive crossings are negative.

The braid isotopy classes of braids on $s$ strands form a group generated by the Artin generators $\sigma_j$, $j=1,2,\ldots,s-1$, where $\sigma_j$ denotes a positive crossing between the strand with the $j$th smallest $\text{Re}(u)$-coordinate (the ``$j$th strand'') and the strand with the next larger $\text{Re}(u)$-coordinate (the $j+1$th strand). The square $B^2$ of a braid $B$ is thus the double repeat of its braid word, two copies of the same braid concatenated.

A braid diagram can be interpreted as a \textit{singular braid} on $s$ strands, that is, a collection of curves that are allowed to intersect transversely in finitely many points, whose image under projection map $(u,t)\mapsto(\text{Re}(u),t)$ into $\mathbb{R}\times[0,2\pi]$ results in a braid diagram (see Figure \ref{fig:diag}). The singular braid monoid is generated by the Artin generators, their inverses $\sigma_j^{-1}$ and $\tau_{j}$, $j=1,2,\ldots,s-1$, which correspond to intersection points between the $j$th strand and the $j+1$th strand. Thus a singular braid in $\mathbb{R}\times[0,2\pi]$ (i.e., it is a braid diagram) is represented by a word that only consists of $\tau_j$s.

The projection map that associates to every geometric braid a braid diagram can thus be understood as a function from the set of braid words to the set of singular braid words mapping a generator $\sigma_j^{\varepsilon}$, $\varepsilon\in\{\pm 1\}$ to $\tau_j$, regardless of the sign $\varepsilon$.

\begin{figure}
\centering
\labellist
\large
\pinlabel a) at 100 1200
\pinlabel b) at 1500 1200
\endlabellist
\includegraphics[height=5cm]{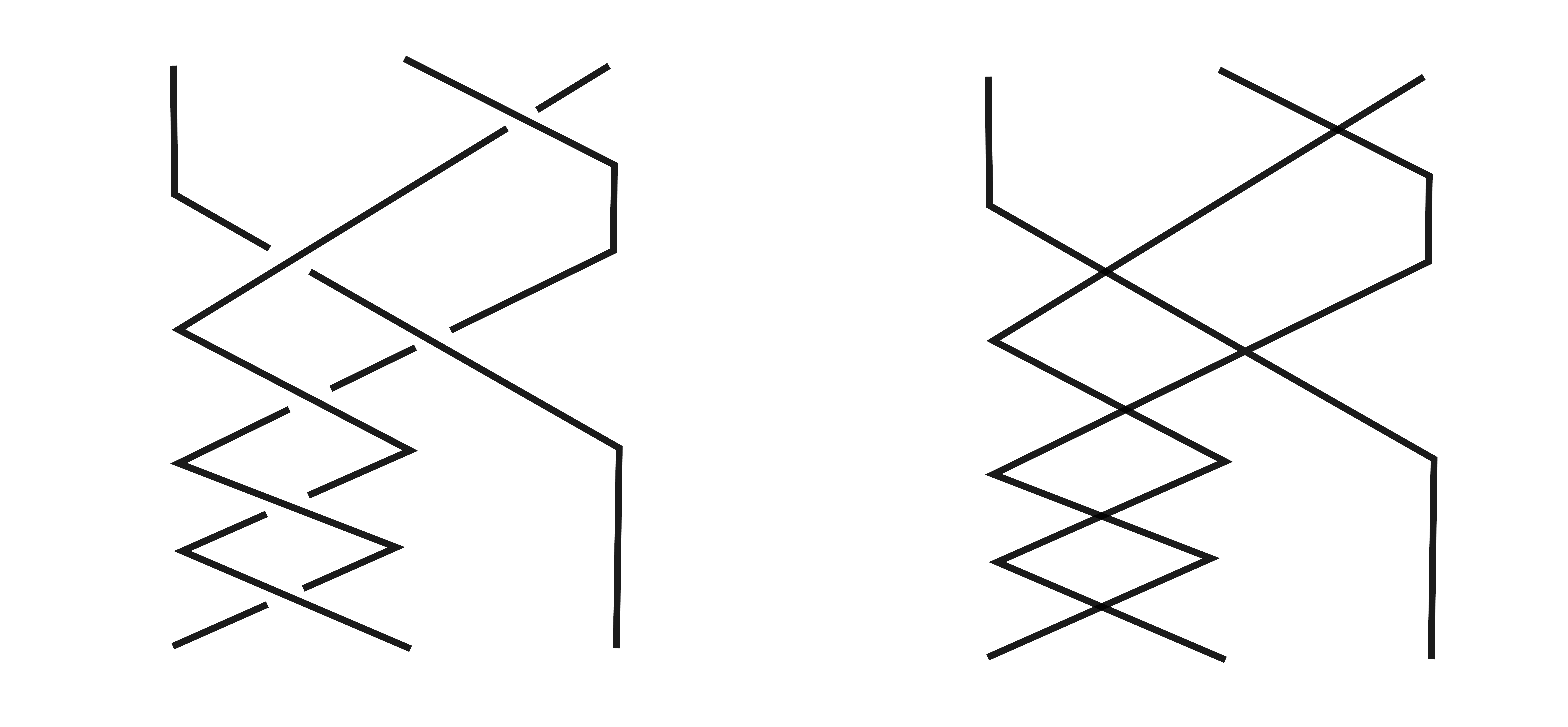}
\caption{a) A braid diagram. b) The corresponding singular braid \label{fig:diag}}
\end{figure}

For some given geometric braid its image under the projection map $(u,t)\mapsto(\text{Re}(u),t)$ into $\mathbb{R}\times[0,2\pi]$ might not be a braid diagram. We call a collection of curves $(u_j(t),t)$, $j=1,2,\ldots,s$, in $\mathbb{C}\times[0,2\pi]$ a \textit{generalized singular braid} if for every $j=1,2,\ldots,s$, there is a unique $i\in\{1,2\ldots,s\}$ with $u_j(0)=u_i(2\pi)$ and all intersection points between the different strands $(u_j(t),t)$ are isolated points. Hence the intersection points are allowed to be tangential and to involve more than two strands. The condition that the intersections are isolated is always satisfied if the strands are not identical and parametrized by real-analytic functions.

We say that a singular crossing of a generalized singular braid is \textit{generic} if it is a transverse intersection between exactly two strands. Otherwise we call it \textit{non-generic}. Non-generic crossings thus consist of more than two strands or are a tangential intersection of at least two strands. We also say that the functions $u_j$ that parametrize a generalized singular braid are non-generic if it has non-generic singular crossings.

\section{An outline of the algorithm}\label{sec:algo}

In \cite{bodepoly} we describe an algorithm that finds for any given braid a semiholomorphic polynomial $f:\mathbb{C}^2\to\mathbb{C}$ whose vanishing set intersects the 3-sphere of unit radius transversely in the closure of the given braid. In \cite{bode:ralg} this construction is modified so that it produces a semiholomorphic polynomial $f$ with a weakly isolated singularity, whose link is the closure of the square $B^2$ of the given braid word $B$.

The first step in both of these constructions is to find (via trigonometric interpolation) a parametrization of the given braid $B$ (up to isotopy) in terms of trigonometric polynomials. Let $\mathcal{C}$ denote the set of connected components of the closure of $B$ and let $s_C$ denote the number of strands that make up the component $C\in\mathcal{C}$. The first step of the algorithm in \cite{bodepoly} finds for every $C\in\mathcal{C}$ a pair of trigonometric polynomials $F_C, G_C:[0,2\pi]\to\mathbb{R}$ such that $B$ is parametrized by
\begin{equation}
\bigcup_{t\in[0,2\pi]}\bigcup_{C\in\mathcal{C}}\bigcup_{j=1}^{s_C}\left(F_C\left(\frac{t+2\pi j}{s_C}\right)+\rmi G_C\left(\frac{t+2\pi j}{s_C}\right),t\right)\subset\mathbb{C}\times[0,2\pi].
\end{equation}

Note that since we use the projection $(u,t)\mapsto(\text{Re}(u),t)$ to obtain braid diagrams and braid words, the real part of the parametrized strands $F_C\left(\tfrac{t+2\pi j}{s_C}\right)$ determines the crossing pattern (i.e., the braid word without the signs of the crossings) and the imaginary part $G_C\left(\frac{t+2\pi j}{s_C}\right)$ determines the signs of the crossings.

In particular, the first step of the algorithm in \cite{bodepoly} yields via trigonometric interpolation a set of trigonometric polynomials $F_C$ such that the corresponding curves $\left(F_C\left(\frac{t+2\pi j}{s_C}\right),t\right)$ parametrize a generalized singular braid.


As in \cite{bodepoly} we would like to point out that the braid diagram for the braid parametrized by the $F_C$s and the $G_C$s is not necessarily identical to the braid diagram of $B$ that we started with. However, the braids are guaranteed to be braid isotopic. The functions $G_C$ are also found via trigonometric interpolation.

Once we have found a parametrization of a braid that is isotopic to $B$ in terms of $F_C$ and $G_C$, we define $g:\mathbb{C}\times S^1$ via 
\begin{equation}
g(u,\rme^{\rmi t})=\prod_{C\in\mathcal{C}}\prod_{j=1}^{s_C}\left(u-F_C\left(\frac{2t+2\pi j}{s_C}\right)-\rmi G_C\left(\frac{2t+2\pi j}{s_C}\right)\right).
\end{equation} 

Note the factor 2 in front of the variable $t$ in the expression above. It means that as $t$ varies between 0 and $2\pi$ we are traversing the braid $B$ twice. In other words, the vanishing set of $g$ is (up to isotopy) the closed braid $B^2$.

Expanding the product above results in a polynomial expression for $g$ with respect to the complex variable $u$, as well as with respect to $\rme^{\rmi 2t}$ and $\rme^{-\rmi 2t}$. We define $p_k:\mathbb{C}^2\to\mathbb{C}$ by
\begin{equation}
p_k(u,r\rme^{\rmi t})=r^{2ks}g\left(\frac{u}{r^{2k}},\rme^{\rmi t}\right),
\end{equation}
where $k$ is a sufficiently large integer. Note that by writing $v=r\rme^{\rmi t}$ this becomes a mixed polynomial $p_k(u,v,\bar{v})=(v\bar{v})^{ks}g\left(\frac{u}{(v\bar{v})^k},\frac{v}{\sqrt{v\bar{v}}}\right)$ if $2ks$ is greater than the degree of $g$ with respect to $\rme^{\rmi t}$ and $\rme^{-\rmi t}$. Note that exponents of $\rme^{\rmi t}$ and $\rme^{-\rmi t}$ in $g$ are even, so that the term $\sqrt{v\bar{v}}$ always comes with an even exponent.

The constructed polynomials $p_k$ are semiholomorphic and radially weighted homogeneous with respect to $P=(2k,1)$ with degree $d(P;f)=2ks$.

Since all roots of $g(\cdot,\rme^{\rmi t})$ are simple, the singularity at the origin is weakly isolated. The link of the singularity is the closure of $B^2$. An explicit isotopy between $p_k^{-1}(0)\cap S_{\rho}^3$ and a projection of $p_{k}^{-1}(0)\cap(\mathbb{C}\times \rho S^1)$ to $S_{\rho}^3$, which is known to be the closure of $B^2$, can be constructed as in \cite{bodepoly}.

\begin{figure}[H]
\labellist
\large
\pinlabel a) at 100 1100
\pinlabel b) at 1000 1100
\pinlabel c) at 2000 1100
\pinlabel d) at 400 -100
\pinlabel e) at 1800 -100
\endlabellist
\includegraphics[height=4.3cm]{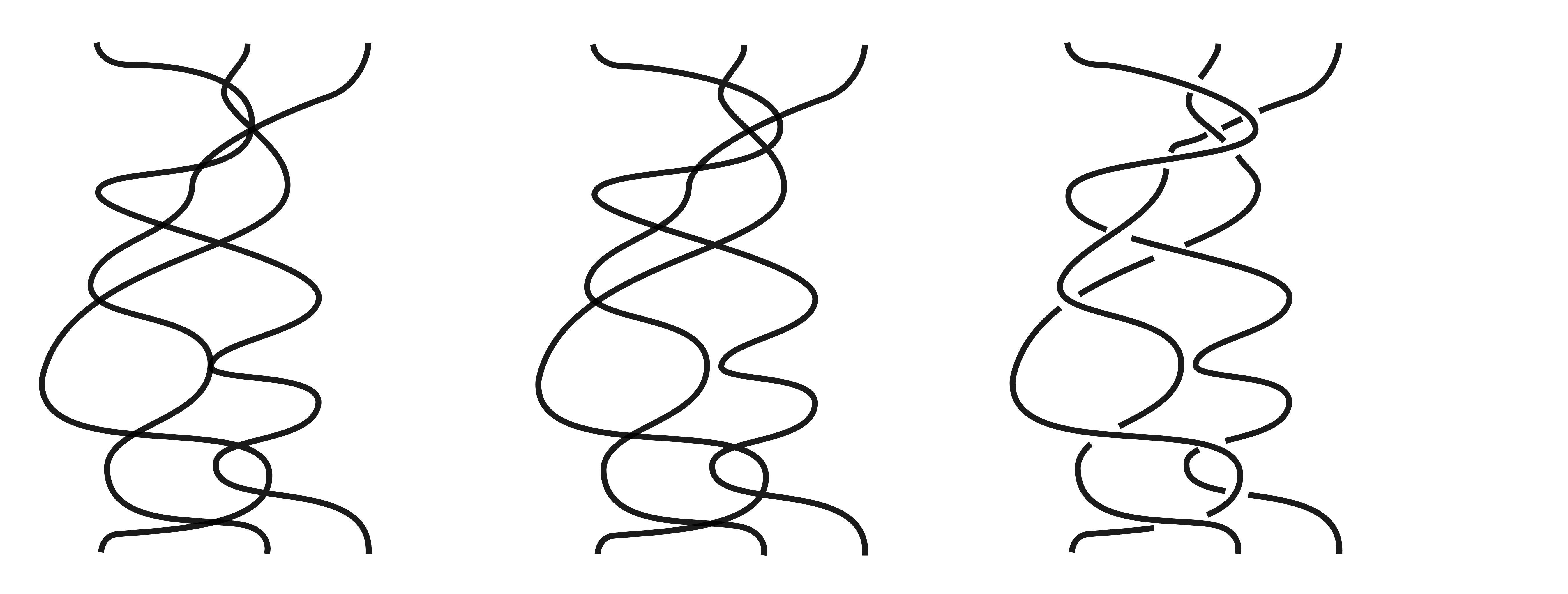}
\vspace{0.5cm}
\includegraphics[height=8.6cm]{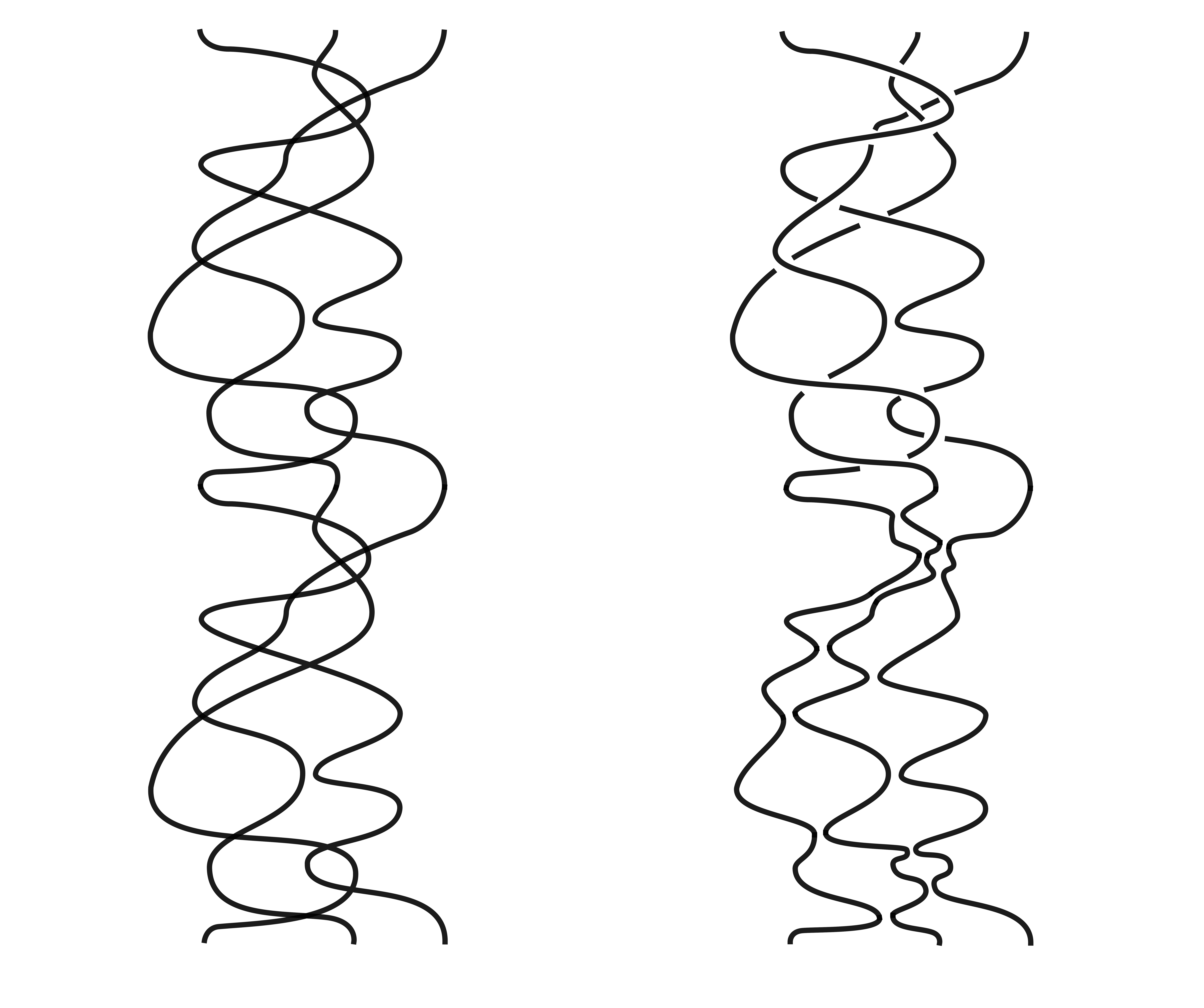}
\caption{a) The generalized singular braid parametrized by the $F_C$s is not necessarily a singular braid. b) We can make the $F_C$s generic. The resulting functions $\tilde{F}_C$ parametrize a singular braid $B_{sing}$. c) The braid diagram of a braid $B'$ that is obtained from an appropriate choice of crossing signs for the singular crossings in $B_{sing}$. The braid $B'$ is braid isotopic to $B$ in Figure \ref{fig:diag}a). d) $B_{sing}^2$, the vanishing set of $g$. e) A resolution of the singular crossings of $B_{sing}^2$ that results in $B'$, whose closure is the link of the singularity of $f$. \label{fig:braids}}
\end{figure}
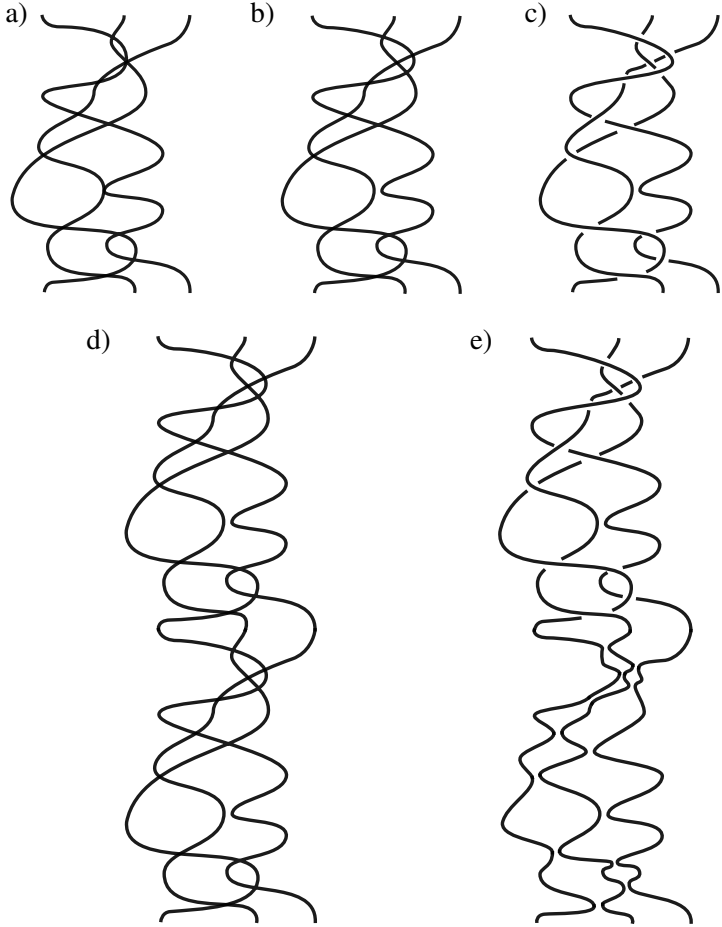

Algorithm 1 below, which constructs a weakly isolated singularity for any given link, is based on the same ideas. However, it uses parametrizations of singular braids instead of classical braids. 
Figure \ref{fig:braids} shows the parametrized braids and vanishing sets of functions at various steps throughout the algorithm.

We start with a braid diagram of a braid $B$ that closes to the link that we want to construct, such as shown in Figure \ref{fig:diag}a). Via the same trigonometric interpolation procedure as in \cite{bodepoly} we obtain trigonometric polynomials $F_C$ that parametrize curves that form a generalized singular braid, see Figure \ref{fig:braids}a). The functions $F_C$ are not necessarily generic. Via small modifications we can make the $F_C$s generic and obtain a singular braid $B_{sing}$ (Figure \ref{fig:braids}b)) that has the property that there exists a choice of signs for each of its singular crossings that turns $B_{sing}$ into a classical braid that is isotopic to $B$, see Figure \ref{fig:braids}c). 

As in \cite{bodepoly} we define a function $p_k$. It is a radially weighted homogeneous mixed polynomial with a singularity at the origin. The intersection $p_k^{-1}(0)\cap S_{\rho}^3$ is the singular braid $B_{sing}^2$ for all $\rho>0$, shown in Figure \ref{fig:braids}d). In particular, it does not have a weakly isolated singularity, since the singular crossings correspond to lines of critical points of $p_k$ through the origin.
However, we can add a term $r^m A(\rme^{\rmi t})$, where $v=r\rme^{\rmi t}$ and $A$ is a finite Fourier series, that makes the singularity weakly isolated. This term has to be constructed in such a way that all singular crossings of $B_{sing}^2$ are resolved in such a way that the link of the singularity is the closure of a braid isotopic to $B$, which is displayed in Figure \ref{fig:braids}e).

\begin{algorithm}[H]
\caption{Construction of weakly isolated singularities}
\begin{steps}
\item From the given braid word $B$ find the trigonometric polynomials $F_C$ via trigonometric interpolation as in \cite{bodepoly}.
\item Make $F_C$ generic. Call the resulting functions $\tilde{F}_C$.
\item Define $g(u,\rme^{\rmi t})=\underset{C\in\mathcal{C}}{\prod}\underset{j=1}{\overset{s_C}{\prod}}\left(u-\tilde{F}_C\left(\tfrac{2t+2\pi j}{s_C}\right)\right)$.
\item Define $p_k(u,r\rme^{\rmi t})=r^{2ks}g\left(\tfrac{u}{r^{2k}},\rme^{\rmi t}\right)$ with $2ks$ greater than the degree of $g$ with respect to $\rme^{\rmi t}$ and $\rme^{-\rmi t}$.
\item Solve the trigonometric interpolation problem $(*)$ in Section \ref{sec:A} for $A:S^1\to\mathbb{C}$.
\item Define $f(u,r\rme^{\rmi t})=p_k(u,r\rme^{\rmi t})+r^mA(\rme^{\rmi t})$, where $m$ is odd and larger than the degree of $A$ with respect to $\rme^{\rmi t}$ and $\rme^{-\rmi t}$ and larger than $2ks$.
\end{steps}
\end{algorithm}

The idea behind Algorithm 1 can be understood as a natural consequence of \cite{eder}, where we introduce certain non-degeneracy conditions of mixed functions and study links of their (weakly) isolated singularities. We show that for such non-degenerate mixed polynomials adding terms above the boundary of the Newton polygon does not change the topology of the link. This seems to suggest that not all links can be obtained as link of weakly isolated singularities of non-degenerate mixed polynomials (and it is an interesting question for which links this is possible). Algorithm 1 thus constructs a degenerate polynomial $p_k$ and adds an appropriate term above the Newton boundary.

In the following sections we explain the individual steps. In particular, we show that each of the steps can be performed algorithmically. Then we show that the algorithm indeed constructs weakly isolated singularities with the closure of $B$ as the link of the singularity.

\section{The individual steps}\label{sec:steps}

Step 1 is identical to the corresponding procedure in \cite{bodepoly}. Note however that the resulting trigonometric polynomials $F_C$ are not necessarily generic. This is not a problem for the construction in \cite{bodepoly}. For the construction in Algorithm 1 however, we need generic parametrizations. This is done in Step 2, which requires a more detailed explanation, detailed in Section \ref{sec:generic}. Step 3 and 4 are then simply definitions of functions. Step 5 is arguably the most important part of this algorithm. It will be discussed in detail in Section \ref{sec:A}. Step 6 is again simply the definition of a function $f$. Thus if Step 2 and Step 5 can be performed algorithmically, Algorithm 1 is indeed an algorithm.

\subsection{Generic parametrizations of singular braids (Step 2)}\label{sec:generic}

The set of trigonometric polynomials $F_C$ that result in generic parametrizations is dense in the set of trigonometric polynomials. So we should expect that the trigonometric polynomials $F_C$ found via the method from \cite{bodepoly} almost always have this property. However, there is no guarantee. If the $F_C$s are not generic, we have to make some adjustments to make them generic. Again we would like to emphasize that in practice, this is usually not necessary. 

Alternative to the method from \cite{bodepoly} trigonometric approximation can be used in Step 1 to find a trigonometric parametrisation of the braid. If the approximated original braid parametrisation is generic, i.e., the corresponding projection gives a braid diagram, then a sufficiently close approximation is generic, too. Therefore, Step 2 of the algorithm is not needed if trigonometric approximation is used in Step 1. However, in contrast to the method from \cite{bodepoly}, trigonometric approximation does not allow us to give bounds on the degrees of the trigonometric polynomials that parametrize the strands.

To a given set of trigonometric polynomials $F_C$, $C\in\mathcal{C}$, with given values of $s_C$, and for any closed interval $[a,b]$ in $[0,2\pi]$ (with $a$ and $b$ away from the values of $t$ for which there are intersections between the $F_C\left(\tfrac{t+2\pi j}{s_C}\right)$s) we call the permutation of the $s=\underset{C\in\mathcal{C}}{\sum}s_C$ curves parametrized by
\begin{equation}
\bigcup_{t\in[a,b]}\bigcup_{C\in\mathcal{C}}\bigcup_{j=1}^{s_C}\left(F_C\left(\frac{t+2\pi j}{s_C}\right),t\right)
\end{equation}
the \textit{permutation associated to the $F_C$s in the interval $[a,b]$}. It is thus an element of the symmetric group on $s$ elements. Note that this is possible, because the $F_C$s are real analytic. This is why even at tangential intersections of strands, we can uniquely determine which incoming arc corresponds to which outgoing arc.

\begin{lemma}[\cite{bodepoly}]\label{lem:perm}
Let $F_C$, $C\in\mathcal{C}$ be trigonometric polynomials and let $B=\underset{j=1}{\overset{\ell}{\prod}}\sigma_{i_j}^{\varepsilon_j}$ be a braid, whose closure has $|\mathcal{C}|$ components. Then there exist trigonometric polynomials $G_C$, $C\in\mathcal{C}$, such that
\begin{equation}
\bigcup_{t\in[0,2\pi]}\bigcup_{C\in\mathcal{C}}\bigcup_{j=1}^{s_C}\left(F_C\left(\frac{t+2\pi j}{s_C}\right)+\rmi G_C\left(\frac{t+2\pi j}{s_C}\right),t\right)
\end{equation}
parametrizes a braid that is braid isotopic to $B$ if there exist values $t_j\in[0,2\pi]$, $j=1,2,\ldots,\ell+1$, $t_1=0$, $t_{\ell+1}=2\pi$, $t_j<t_{j+1}$ such that the permutation associated to the $F_C$s in the interval is $[t_j,t_{j+1}]$ is the transposition $(i_j\leftrightarrow i_j+1)$. 
\end{lemma}

The algorithm in \cite{bodepoly} finds trigonometric polynomials $F_C$ such that the condition in Lemma \ref{lem:perm} is satisfied. We would like to make the $F_C$s generic, while maintaining this property.

The $F_C$s being non-generic could mean that there are tangential intersections between strands of the corresponding generalized singular braid $B_{sing}$ or that there are more than two strands involved in a singular crossing of $B_{sing}$.

\begin{figure}
\centering
\labellist
\large
\pinlabel a) at 100 1600
\pinlabel b) at 1900 1600
\pinlabel c) at 100 800
\pinlabel d) at 1900 800
\endlabellist
\includegraphics[height=6cm]{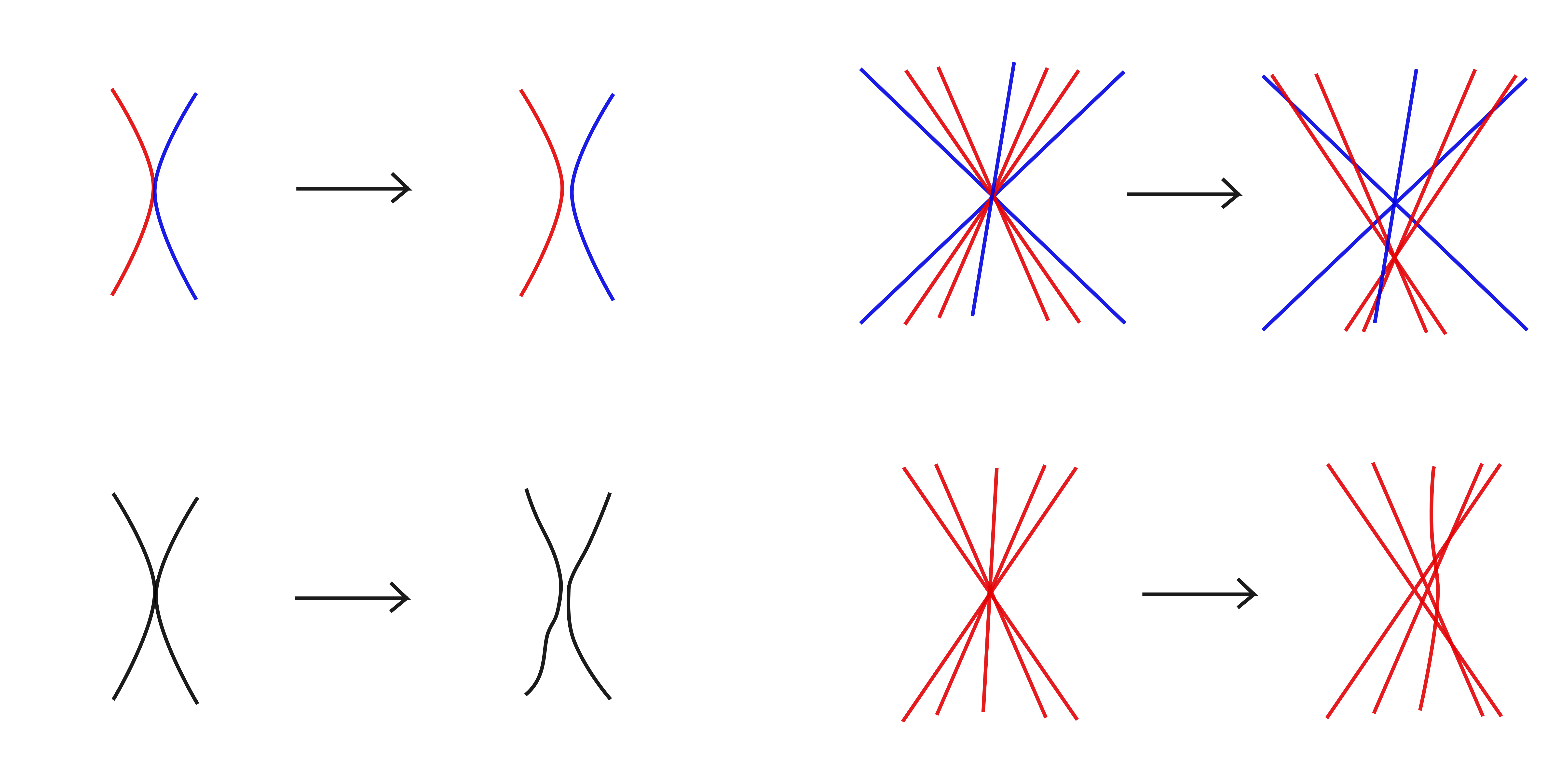}
\caption{The elimination of non-generic crossings. a) A tangential intersection between strands from different components gets eliminated. b) An intersection between more than 2 strands from different components gets eliminated. c) A tangential intersection between strands from the same component gets eliminated. d) An intersection between more than two strands, all of which are from the same component, is eliminated. \label{fig:elim}}
\end{figure}
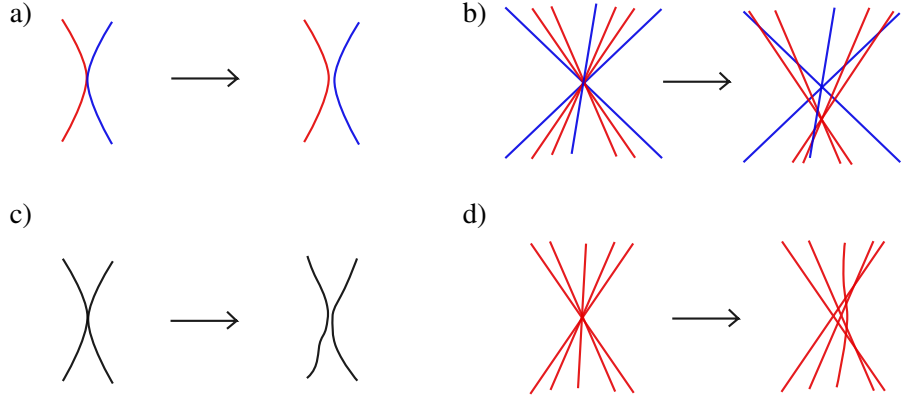

Having explicit expressions for the functions $F_C$ we can find all values $t=t_{k}'$, $k=1,2,\ldots,M$, for which there are non-generic crossings. The fact that there are only finitely many of these follows from the real analyticity of the functions. It will simplify our notation if we adopt the convention of updating our variables throughout the modifications outlined below. That is to say, when we change the function $F_C$ for example by adding a term, the resulting function will again be called $F_C$. The values $t_k'$, $k=1,2,\ldots,M$, are again defined as the values of $t$ at which the new collection of functions $F_C$ has non-generic crossings. Note that their number $M$ can change throughout our modification, and will eventually be 0.

A tangential intersection between the strands $(C,j)$ and $(C',j')$ at $t=t_k'$ corresponds to a non-simple root of $F_{C}\left(\tfrac{t+2\pi j}{s_C}\right)-F_{C'}\left(\tfrac{t+2\pi j'}{s_{C'}}\right)$ at $t=t_k'$ and a singular crossing between more than two strands $(C_i,j_i)$, $i=1,2,\ldots,m'$, at $t=t_k$ corresponds to a common root of $F_{C_i}\left(\tfrac{t+2\pi j_i}{s_{C_i}}\right)-F_{C_{i'}}\left(\tfrac{t+2\pi j_{i'}}{s_{C_{i'}}}\right)$ at $t=t_k'$. We can thus check numerically if any given collection of trigonometric parametrizations $F_C$ is generic or not. Likewise, we can check numerically if the $F_C$s satisfy the condition from Lemma \ref{lem:perm} for the same values $t_j$, $j=1,2,\ldots,\ell+1$, as the original functions $F_C$.

We can remove the tangential intersection points between two strands of different components $C$ and $C'$ by adding small constants $\varepsilon_{C,1}$ to each $F_C$. This is displayed in Figure \ref{fig:elim}a). This requires $\varepsilon_{C,1}\neq\varepsilon_{C',1}$ if $C\neq C'$. Furthermore, we can choose $\varepsilon_{C,1}$ sufficiently small so that the resulting trigonometric polynomials $F_C$ still satisfy the property from Lemma \ref{lem:perm} for the same values $t_j$, $j=1,2,\ldots,\ell+1$. Furthermore, the addition of $\varepsilon_{C,1}$ should not introduce any new non-generic crossings. This can be achieved by choosing the values for the different $\varepsilon_{C,1}$s successively, i.e., for an arbitrary ordering of the components $C_1,C_2,\ldots,C_{|C|}$ we first choose $\varepsilon_{C_1,1}$ such that it removes tangential intersection points involving strands from $C_1$ without introducing new ones, then we choose $\varepsilon_{C_2,1}$ and so on. We also add a small constant to every $F_C$ that is constant. Such components only consist of a single vertical strand. Adding a small constant guarantees that none of them are involved in any non-generic crossings.

Note that sufficient values for $\varepsilon_{C,1}$ can be found explicitly knowing the values of $t$ for which we have generic or non-generic crossings as well as maxima and minima of the functions $F_{C}\left(\tfrac{t+2\pi j}{s_C}\right)-F_{C'}\left(\tfrac{t+2\pi j'}{s_{C'}}\right)$, $C,C'\in\mathcal{C}$, $j\in\{1,2,\ldots,s_C\}$, $j'\in\{1,2,\ldots,s_{C'}\}$. Alternatively, since we can check numerically if the $F_C$s are generic or not, we can take $\varepsilon_{C,1}$ to be an element of a non-zero sequence converging to $0$ and if the resulting $F_C$ is non-generic, we redefine $\varepsilon_{C,1}$ to be the next element in that sequence.

By taking $F_C(t+\varepsilon_{C,2})$ instead of $F_C(t)$ as the trigonometric polynomial for the component $C$ with appropriately chosen small $\varepsilon_{C,2}$, we obtain a parametrization where every singular crossing that involves more than two strands only involves strands from the same component. This is achieved by choosing $\varepsilon_{C,2}\neq\varepsilon_{C',2}$ if $C\neq C'$ and each $\varepsilon_{C,2}$ sufficiently small. We do not introduce any new non-generic crossing in doing this, since the intersection is transverse, the curves are real analytic and the intersection does not involve any constant strands. The effect is shown in \ref{fig:elim}b). How small we have to choose each $\varepsilon_{C,2}$ can be determined from the values of $t$ for which there are crossings. Note in particular that the $\varepsilon_{C,2}$s can be chosen such that the condition from Lemma \ref{lem:perm} is still satisfied for the same values $t_j$, $j=1,2,\ldots,\ell+1$, as before. Note that we find the values of $\varepsilon_{C,2}$ successively. We choose a value for one component $C$ and only then decide on the value for the next component $C'$ and so on.

Thus the only remaining non-generic crossings are between strands of the same component. Suppose we have a tangential intersection between $(C,j)$ and $(C,j')$ at $t=t_k'$. Then we add $\varepsilon\cos\left(t-\frac{t_k'+2\pi j}{s_C}\right)$ to $F_C$, where as usual $\varepsilon$ is small and its sign is determined by the sign of $F_{C}\left(\tfrac{t+2\pi j}{s_C}\right)-F_{C}\left(\tfrac{t+2\pi j'}{s_{C}}\right)$ in a neighbourhood of $t=t_k'$. 


By adding $\varepsilon\cos\left(t-\frac{t_k+2\pi j}{s_C}\right)$ we know that $F_{C}\left(\tfrac{t+2\pi j}{s_C}\right)-F_{C}\left(\tfrac{t+2\pi j'}{s_{C}}\right)$ is non-zero in a neighbourhood of $t=t_k'$ (independent of $\varepsilon$) for all $\varepsilon$ of the correct sign and sufficiently small modulus. We have thus reduced the number of tangential intersections by one, see Figure \ref{fig:elim}c). Proceeding like this we eliminate all tangential intersections between strands. This includes tangential intersections that are part of non-generic crossings with more than two strands. Again we can choose $\varepsilon$ sufficiently small, so that Lemma \ref{lem:perm} is still satisfied for $t_j$, $j=1,2,\ldots,\ell+1$.

Thus the only remaining non-generic crossings are crossings that involve more than two strands and all of them are from the same component. Suppose we have such a crossing at $t=t_k'$ and two of the strands involved in that crossing are $(C,j)$ and $(C,j')$. Then we add $\varepsilon'\cos\left(t-\varphi\right)$, where $\varphi=t_k'/s_C-\pi+\pi(j'-j)/s_C$. This value is chosen such that $\cos\left(\tfrac{t_k'+2\pi j}{s_C}-\varphi\right)=\cos\left(\tfrac{t_k'+2\pi j'}{s_C}-\varphi\right)$, while $\cos\left(\tfrac{t_k'+2\pi j}{s_C}-\varphi\right)\neq\cos\left(\tfrac{t_k'+2\pi j''}{s_C}-\varphi\right)$ for all $j''\notin\{j,j'\}$. Thus after adding $\varepsilon'\cos\left(t-\varphi\right)$ we still have a crossing at $t=t_k'$ between $(C,j)$ and $(C,j')$, but no other strand is involved in that crossing. Hence it is a generic crossing. The other strands that used to be part of that crossing have been moved aside and could form other non-generic crossings with more than two strands. Therefore, we have not necessarily reduced the number of non-generic crossings in this step. However, we have reduced the sum of the number of strands involved in a non-generic crossing $c$, with the sum going over all non-generic crossings $c$. Thus repeating this step we can eliminate all non-generic crossings and obtain a generic parametrization $F_C$. This elimination process is illustrated in Figure \ref{fig:elim}d).

It is more difficult to give an explicit formula for a sufficient value of $\varepsilon'$. Since we are not particularly concerned with achieving an optimal run-time for our algorithm, we may again resort to the approach of using a non-zero sequence converging to $0$ and checking at each value of $\varepsilon'$ if the resulting parametrisation is generic and satisfies Lemma \ref{lem:perm}.

Let now $\tilde{F}_C$, $C\in\mathcal{C}$ denote the generic trigonometric polynomials that we obtain from this procedure and let $B_{sing}=\underset{j=1}{\overset{\ell'}{\prod}}\tau_{i_j}$ be the singular braid that is parametrized by the $\tilde{F}_C$s. 
\begin{lemma}\label{lem:signs}
There is a choice of signs $\varepsilon_j\in\{\pm 1\}$ such that the input braid $B$ is braid isotopic to $\underset{j=1}{\overset{\ell'}{\prod}}\sigma_{i_j}^{\varepsilon_j}$.
\end{lemma}
\begin{proof}
We have selected the values of the different $\varepsilon_{C,1}$s, $\varepsilon_{C,2}$s, $\varepsilon$s and $\varepsilon'$s such that the $\tilde{F}_C$s still satisfy the condition from Lemma \ref{lem:perm} for the same values $t_j$, $j=1,2,\ldots,\ell+1$ as the original $F_C$s. Therefore by Lemma \ref{lem:perm}, there exist trigonometric polynomials $\tilde{G}_C$, $C\in\mathcal{C}$, such that $\tilde{F}_C+\rmi \tilde{G}_C$ parametrizes a braid $B'$ that is braid isotopic to $B$. Since the $\tilde{F}_C$s are generic, this is equivalent to $B_{sing}$ being obtained from a braid diagram of $B'$ by forgetting information about signs of crossings, with $B'$ a braid that is isotopic to $B$. Thus the value of $\varepsilon_j$ is the sign of the corresponding crossing in the braid $B'$.
\end{proof}

Note that throughout Step 2 we only add terms of degree 0 or 1 with respect to $\rme^{\rmi t}$ and $\rme^{-\rmi t}$ and the degree 1 terms are only necessary for components with more than one strand. Therefore, the degrees of the trigonometric polynomials $F_C$ are not affected by the procedure above and the degree of $\tilde{F}_C$ is equal to the degree of $F_C$.

\subsection{Trigonometric interpolation (Step 5)}\label{sec:A}

For a given generic collection of trigonometric polynomials $\tilde{F}_C$ the roots of 
\begin{equation}\label{eq:defg}
g(u,\rme^{\rmi t})=\underset{C\in\mathcal{C}}{\prod}\underset{j=1}{\overset{s_C}{\prod}}\left(u-\tilde{F}_C\left(\tfrac{2t+2\pi j}{s_C}\right)\right)
\end{equation} 
form a singular braid $B_{sing}^2$ that is the square of the singular braid $B_{sing}$. Both of these singular braids have $s=\underset{C\in\mathcal{C}}{\sum}s_C$ strands.

Let $t_k$, $k=1,2,\ldots,\ell'$ denote the values of $t\in[0,2\pi]$ for which there are singular crossings in $B_{sing}$. By a shift of the variable $t$, we can always guarantee that $t_k\neq \pi$ for all $k$. 
Denote by $(C_1(k),j_1(k))$ and $(C_2(k),j_2(k))$ the two strands that form the crossing at $t=t_k$. Which of these strands carries which label is not important, but by convention we choose the labels such that $\tilde{F}_{C_1(k)}\left(\tfrac{t+2\pi j_1(k)}{s_{C_1(k)}}\right)<\tilde{F}_{C_2(k)}\left(\tfrac{t+2\pi j_2(k)}{s_{C_2(k)}}\right)$ for all $t\in(t_k-\varepsilon,t_k)$ for some small $\varepsilon>0$. 

Note that $g$ from Eq.~\eqref{eq:defg} has only real roots and is therefore a real polynomial. Since all roots of $g(\cdot,\rme^{\rmi t})$ are simple when $t\neq t_k$, $k=1,2,\ldots,\ell'$, there is a critical point of $g$ between each neighbouring pair of roots of $g$, i.e., if $u_1,u_2\in\mathbb{R}$ are roots of $g(\cdot,\rme^{\rmi t})$ and there is no root of $g(\cdot,\rme^{\rmi t})$ in the open interval $(u_1,u_2)$, there is a unique critical point $c\in(u_1,u_2)$. We call $\sign(g(c,\rme^{\rmi t}))$ the sign of the critical point $c$. As $t$ varies, the critical points of $g(\cdot,\rme^{\rmi t})$ move on the real line, but they remain distinct and maintain their sign for all $t\neq t_k$, $k=1,2,\ldots,\ell'$.

At $t=t_k$ two roots and their intermediate critical point $c$ collide. We say that $c$ is the critical point associated with the crossing.

Step 5 of Algorithm 1 is to solve the following trigonometric interpolation problem $(*)$:\\
The set of data points takes the form $(t_k,y_k,z_k)$, $k=1,2,\ldots,\ell'$, where $t_k$, $k=1,2,\ldots,\ell'$, are as above the values of $t$ for which there are crossings of $B_{sing}$. The value $y_k$ is such that $\tfrac{y_k}{\cos\left(\tfrac{t_k}{2}\right)}$ is a non-zero real number that has the same sign as the critical point associated with the crossing at $t=t_k$.

We know from Lemma \ref{lem:signs} that for every crossing of $B_{sing}=\underset{k=1}{\overset{\ell'}{\prod}}\tau_{j_k}$ there is a choice of sign $\varepsilon_k\in\{\pm 1\}$ such that $B'=\prod_{k=1}^{\ell'}\sigma_{j_k}^{\varepsilon_{k}}$ is braid isotopic to $B$ and thus closes to the desired link. The value of $z_k$ is set to $\varepsilon_k$.

\ \\

\textbf{The interpolation problem $(*)$}: Find a trigonometric polynomial $\tilde{A}:S^1\to\mathbb{C}$ such that $\tilde{A}(\rme^{\rmi t_k})=\tfrac{y_k}{\cos\left(\tfrac{t_k}{2}\right)}$ and $\tfrac{\partial \arg(\tilde{A})}{\partial t}(\rme^{\rmi t_k})=z_k$ for all $k\in\{1,2,\ldots,\ell'\}$. 

\ \\

Since 
\begin{equation}
\frac{\partial \arg(\tilde{A})}{\partial t}(\rme^{\rmi t_k})=\left.\left(\frac{\text{Re}(\tilde{A})\frac{\partial \text{Im}(\tilde{A})}{\partial t}-\text{Im}(\tilde{A})\frac{\partial \text{Re}(\tilde{A})}{\partial t}}{|\tilde{A}|^2}\right)\right|_{t=t_k},
\end{equation} the interpolation problem above can be written as an interpolation where the values of the data points correspond to values of the desired function $\tilde{A}$ and its first derivative. Such an interpolation always has a solution that can be found via explicit formulas such as the one in \cite{nathan}. The degree of the solution is equal to $\ell'$.

We then define $A(\rme^{\rmi t}):=\tilde{A}(\rme^{\rmi 2t})\cos(t)$, which satisfies $A(\rme^{\rmi t_k/2})=y_k$ and 
$\tfrac{\partial \arg(A)}{\partial t}(\rme^{\rmi t_k/2})$ has the same sign as $\varepsilon_k$ for all $k\in\{1,2,\ldots,\ell'\}$.

Since $A$ is odd, i.e., $A(\rme^{\rmi (t+\pi)})=-A(\rme^{\rmi t})$, it automatically also satisfies $A(\rme^{\rmi (t_{k}/2+\pi)})=-y_k$ and $\tfrac{\partial \arg(A)}{\partial t}(\rme^{\rmi (t_{k}/2+\pi)})$ also has the same sign as $\varepsilon_k$ for all $k\in\{1,2,\ldots,\ell'\}$.

\section{Weakly isolated singularities}\label{sec:proof}
In this section we prove that Algorithm 1 does what it is supposed to do: It constructs weakly isolated singularities with the desired link as the link of the singularity. Thereby we establish a proof of Theorem \ref{thm:ak}. 

We use the same notation as in the previous sections. $\tilde{F}_C$ is a generic trigonometric parametrization of the singular braid $B_{sing}=\underset{k=1}{\overset{\ell'}{\prod}}\tau_{j_k}$. Let $\varepsilon_k\in\{\pm 1\}$, $k=1,2,\ldots,\ell'$ and let $A:S^1\to\mathbb{C}$ be the trigonometric polynomial found via the interpolation procedure in Step 5 of Algorithm 1. Let $g$, $p_k$ and $f=p_k+r^mA(\rme^{\rmi t})$ be defined as in the description of Algorithm 1.

\begin{lemma}
For every fixed and sufficiently small $r_*>0$ the vanishing set of $f|_{r=r_*}:\mathbb{C}\times S^1\to\mathbb{C}$ is the closed braid $\underset{k=1}{\overset{\ell'}{\prod}}\sigma_{j_k}^{\varepsilon_k}$.
\end{lemma}
\begin{proof}
The vanishing set of $f|_{r=r_*}$ corresponds (up to an overall scaling in the $u$-coordinate) to the vanishing set of $g+r_*^{m-2ks}A$, which is equal to $(g_t)^{-1}(-r_*^{m-2ks}A(\rme^{\rmi t}))$.

Since $g_t$ is monic and real and its critical points are distinct for all values of $t\in[0,2\pi]$, there is a diffeomorphism $h:\mathbb{C}\times S^1\to\mathbb{C}\times S^1$ that is the identity outside of $\{(u,\rme^{\rmi t}):|u|<R\}$ for some $R>0$ and that preserves the fibers of the projection map onto the second factor $\mathbb{C}\times S^1\to S^1$, and a disk $D$ such that $(g_t(h))^{-1}(\mathbb{R})\cap D$ is the union of the real line ($\{(u,\rme^{\rmi t}):\text{Im}(u)=0\}\cap D$) and $s-1$ straight, ``vertical'' lines orthogonal to the real line for every $t\in[0,2\pi]$. This is displayed in Figure \ref{fig:fol}. Note that the vertical lines intersect the real line in the critical points of $g_t$. Since the critical points vary with $t$, so do the vertical lines.

\begin{figure}
\centering
\labellist
\large
\pinlabel $\mathbb{R}$ at 1950 1150
\pinlabel $D$ at 1250 550
\pinlabel $(g_t(h))^{-1}(\mathbb{R})$ at -200 1600
\endlabellist
\includegraphics[height=6cm]{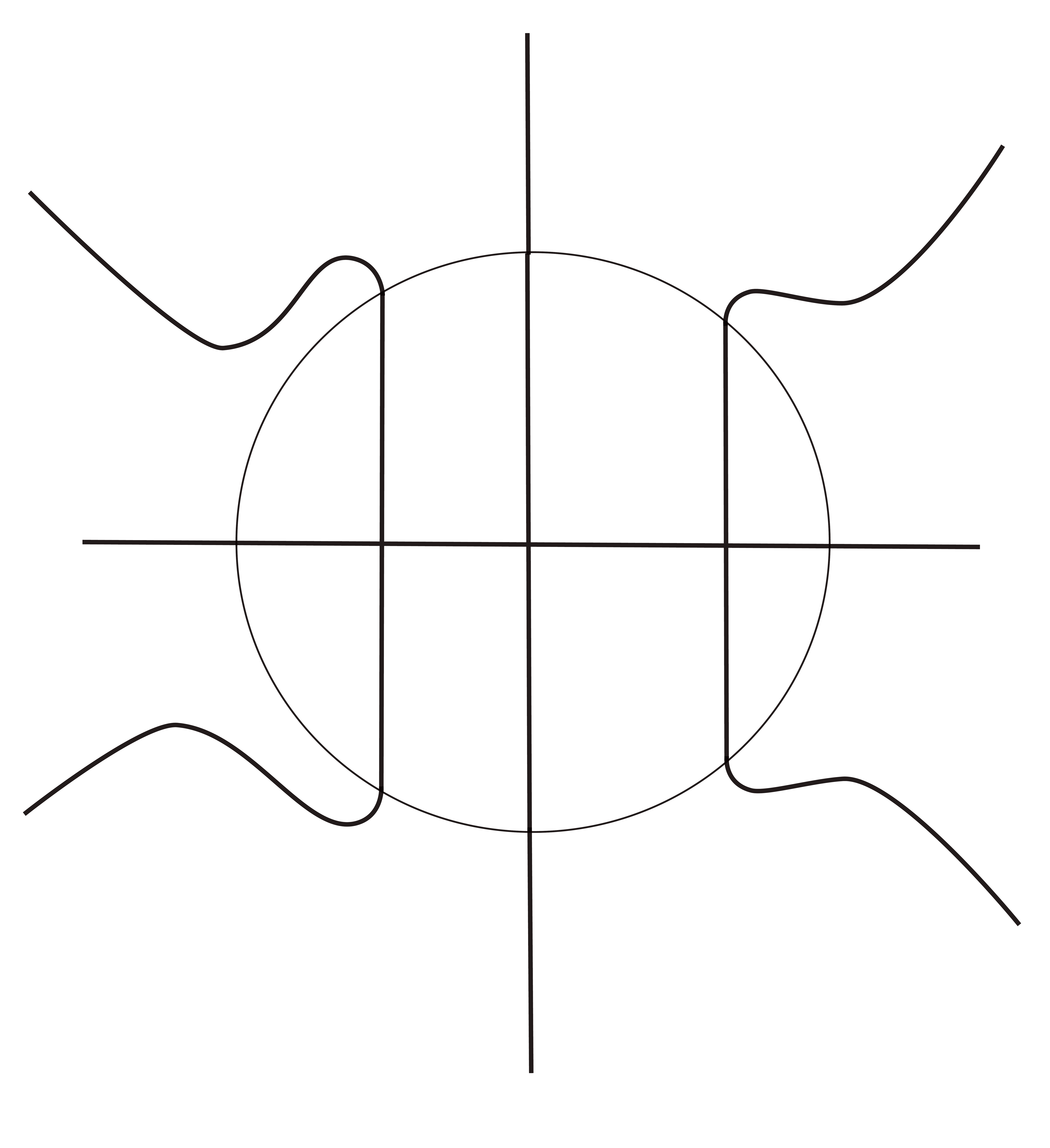}
\caption{The preimage $(g_t(h))^{-1}(\mathbb{R})$ in the complex plane for a fixed value of $t\in[0,2\pi]$ and the disk $D$.\label{fig:fol}}
\end{figure}

Let $t_k$, $k=1,2,\ldots,2\ell'$, denote the values of $t$ for which there are crossings of $B_{sing}^2$. Note that for $k\leq\ell'$ these differ from the values of $t_k$ in the previous section, corresponding to the crossings of $B_{sing}$, by a factor of $1/2$. By symmetry we have $t_{k+\ell'}=t_k+\pi$. Figure \ref{fig:motion} shows subsets of the complex plane in a neighbourhood of a singular crossing at $t=t_k$, and at $t=t_k+\pi$. The black lines are the preimage set $(g_t(h))^{-1}(\mathbb{R})$ with the horizontal line being a segment of the real line. The red points are the roots of $g_t$ at values $t=t_k-2\varepsilon$, $t=t_k-\varepsilon$, $t=t_k$, $t=t_k+\varepsilon$ and $t=t_k+2\varepsilon$. By symmetry the corresponding roots at $t_{k+\ell'}$ are the same. The blue points indicate the roots of $g_t(h)+\delta A(\rme^{\rmi t})$, which are the preimage points $(g_t(h))^{-1}(-\delta A(\rme^{\rmi t}))$, for some small $\delta>0$.

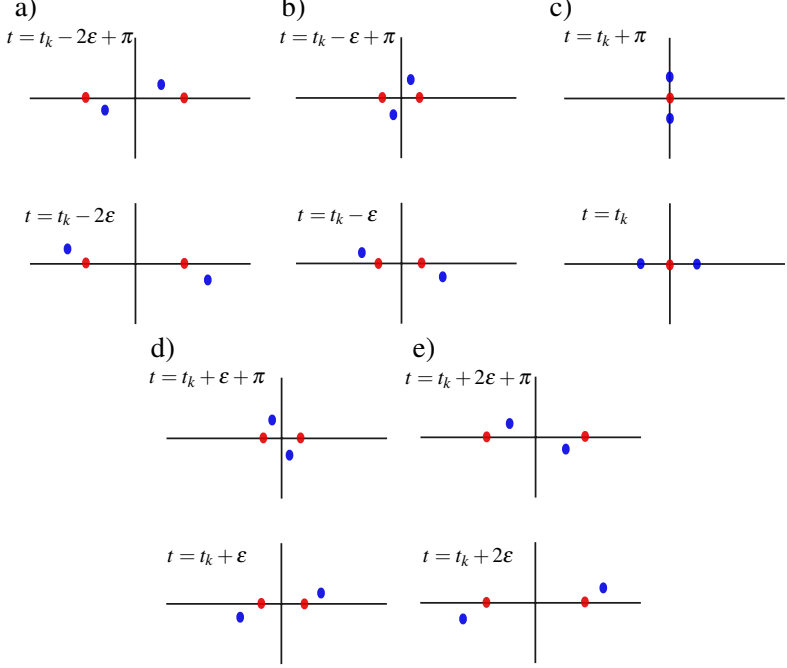
\begin{figure}
\centering
\labellist
\large
\pinlabel a) at 100 2400
\footnotesize 
\pinlabel $t=t_k-2\varepsilon+\pi$ at 400 2200
\pinlabel $t=t_k-2\varepsilon$ at 400 1000
\endlabellist
\includegraphics[height=4.5cm]{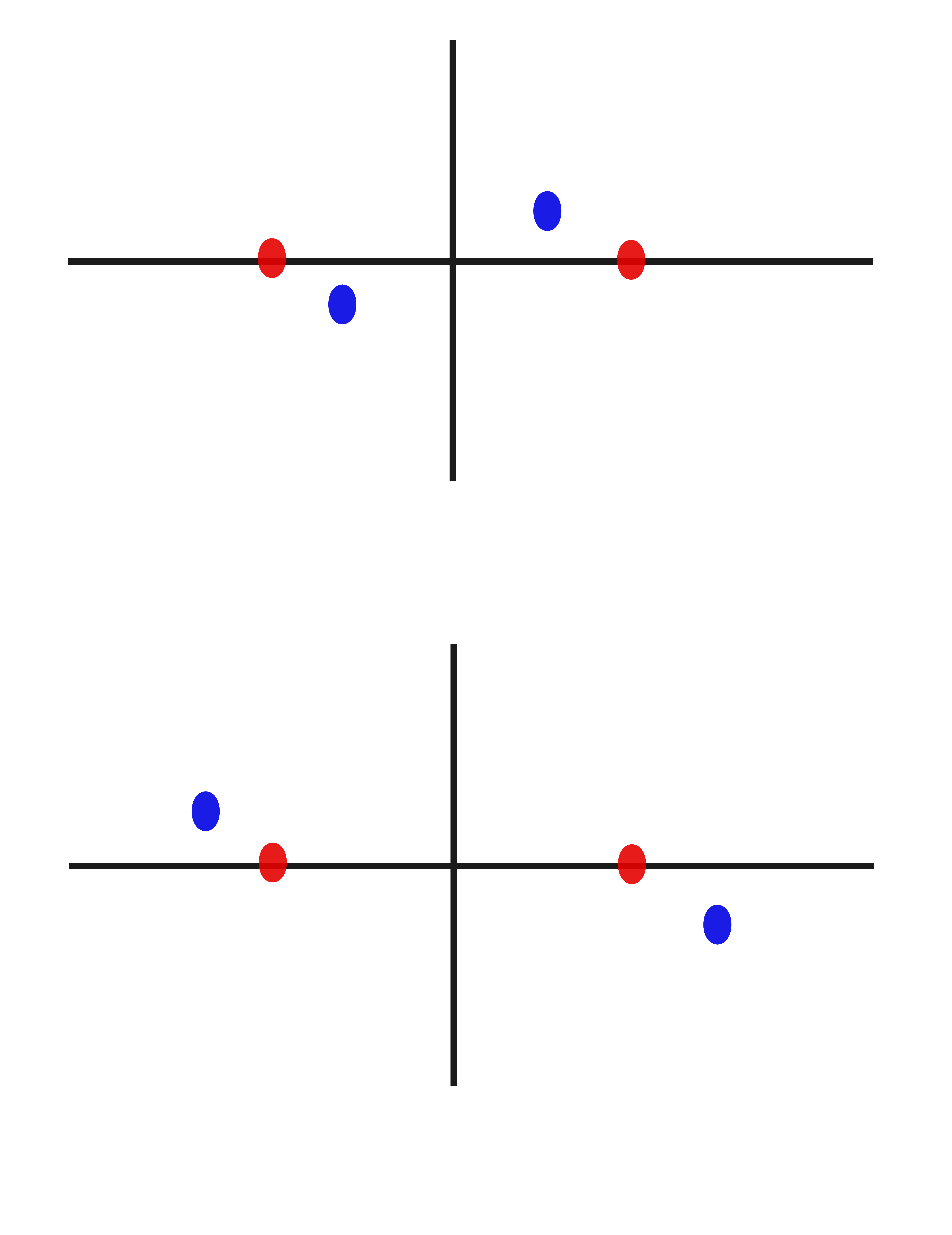}
\labellist
\large
\pinlabel b) at 100 2400
\footnotesize 
\pinlabel $t=t_k-\varepsilon+\pi$ at 400 2200
\pinlabel $t=t_k-\varepsilon$ at 400 1000
\endlabellist
\includegraphics[height=4.5cm]{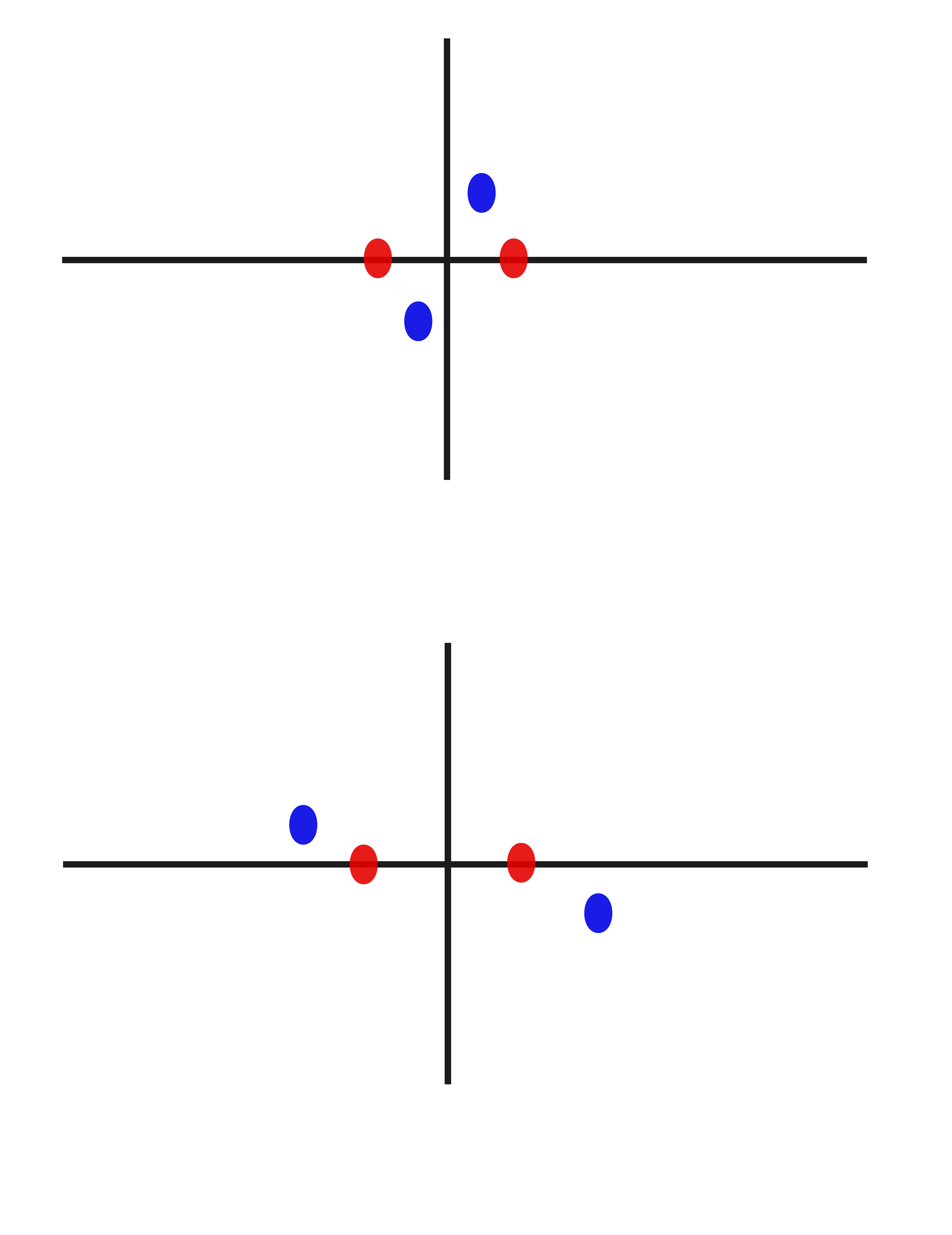}
\labellist
\large
\pinlabel c) at 100 2400
\footnotesize 
\pinlabel $t=t_k+\pi$ at 400 2200
\pinlabel $t=t_k$ at 400 1000
\endlabellist
\includegraphics[height=4.5cm]{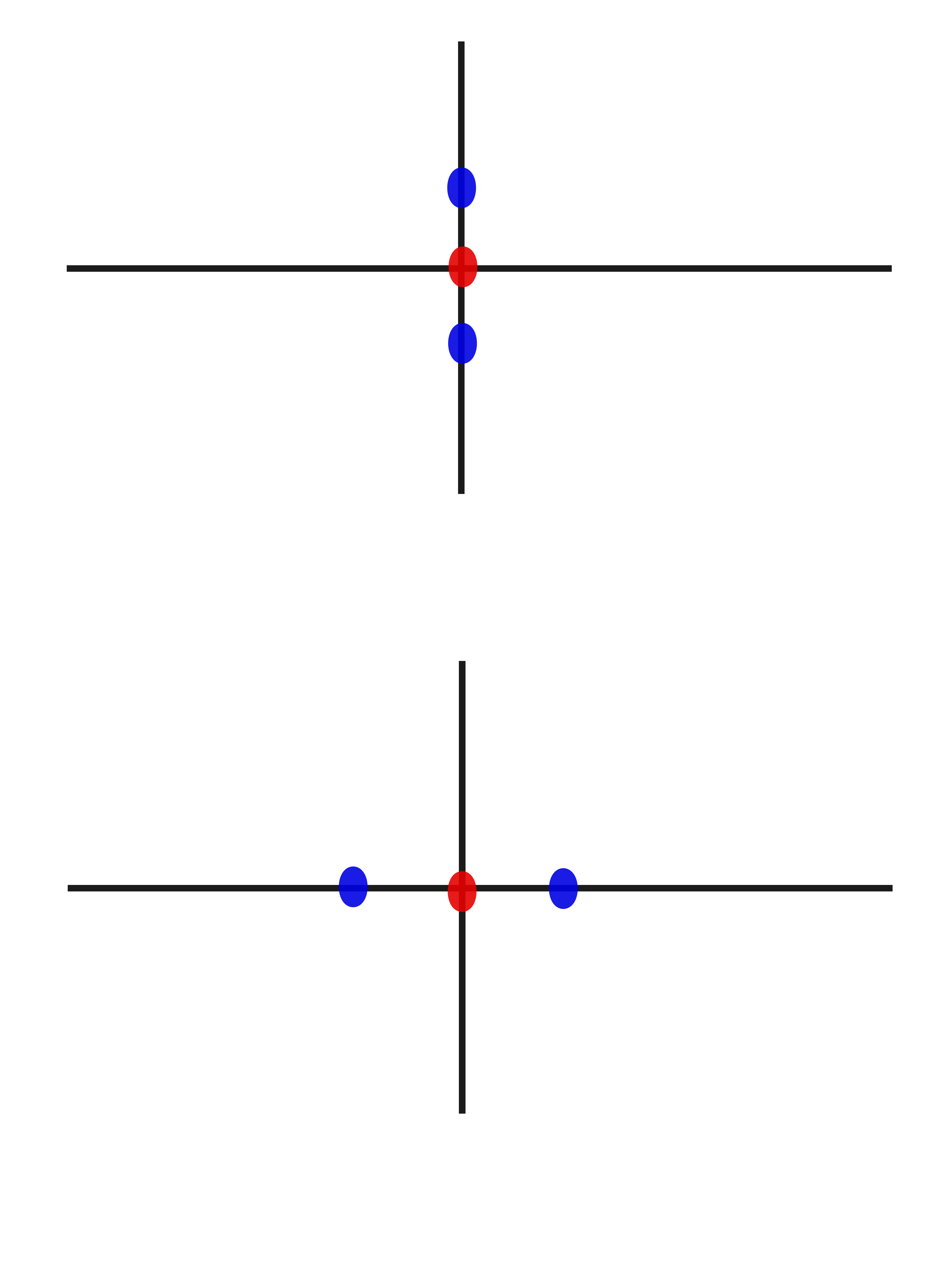}
\labellist
\large
\pinlabel d) at 100 2400
\footnotesize 
\pinlabel $t=t_k+\varepsilon+\pi$ at 400 2200
\pinlabel $t=t_k+\varepsilon$ at 400 1000
\endlabellist
\includegraphics[height=4.5cm]{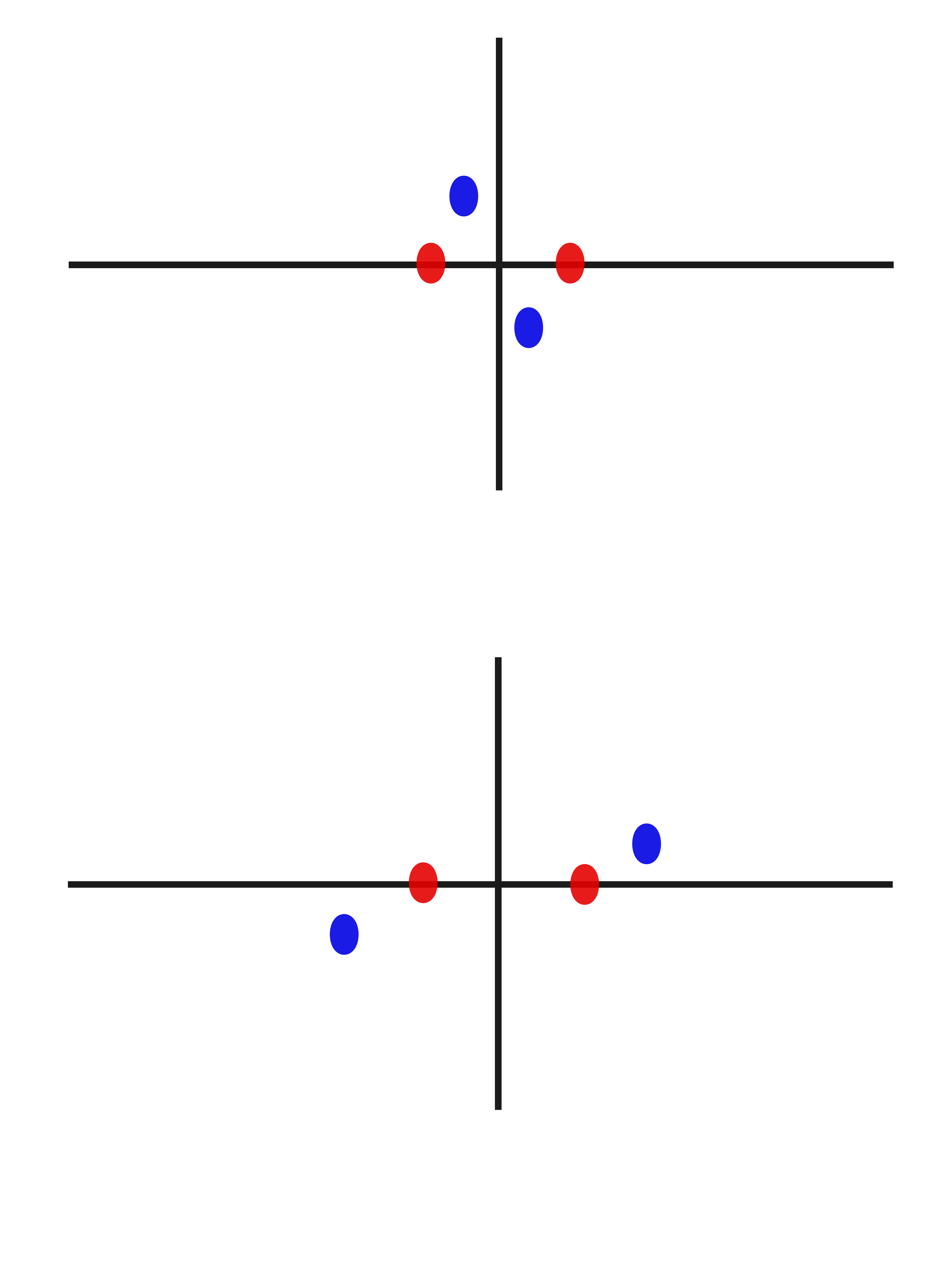}
\labellist
\large
\pinlabel e) at 100 2400
\footnotesize 
\pinlabel $t=t_k+2\varepsilon+\pi$ at 400 2200
\pinlabel $t=t_k+2\varepsilon$ at 400 1000
\endlabellist
\includegraphics[height=4.5cm]{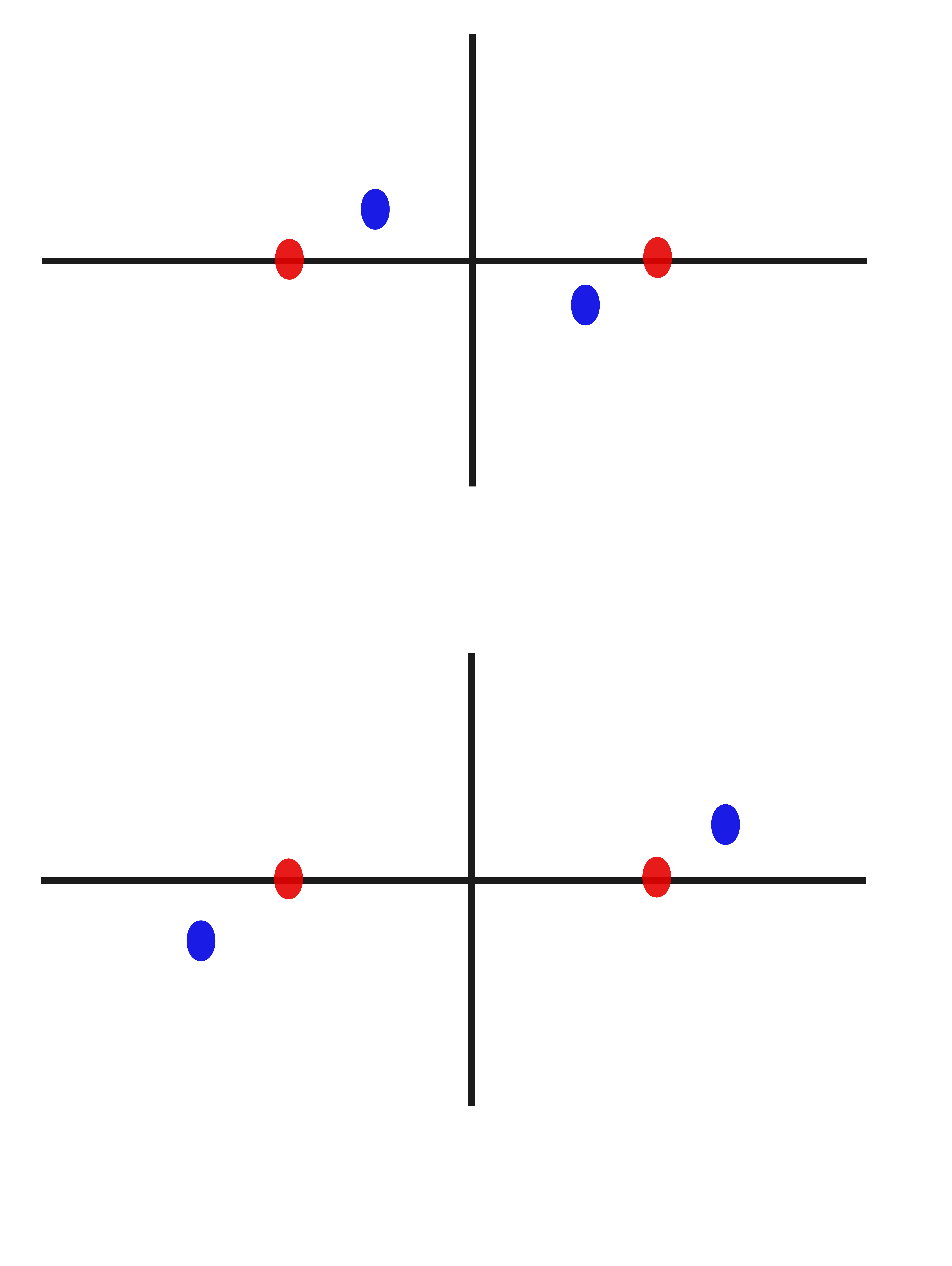}
\caption{The motion of the roots of $g_t$ (in red) and of $g_t+r^{m-2ks}A$ (in blue) in the complex plane in a neighbourhood of singular crossings at $t=t_k$ and $t=t_k+\pi$, $k\in\{1,2,\ldots,\ell'\}$. For each subfigure the lower part shows the behaviour near $t=t_k$ and the upper part shows the behaviour near $t=t_k+\pi$. a) At $t=t_k-\varepsilon$ and $t=t_k+\pi-\varepsilon$. b) At $t=t_k-\varepsilon/2$ and $t=t_k+\pi-\varepsilon/2$. c) At $t=t_k$ and $t=t_k+\pi$. d) At $t=t_k+\varepsilon/2$ and $t=t_k+\pi+\varepsilon/2$. e) At $t=t_k+\varepsilon$ and $t=t_k+\pi+\varepsilon$. \label{fig:motion}}
\end{figure}

By construction $A(\rme^{\rmi t_k})$, $k=1,2,\ldots,\ell'$, is real and has the same sign as the critical point associated with the crossing at $t=t_k$. Hence the two preimage points $(g_{t_k}(h))^{-1}(-\delta A(\rme^{\rmi t_k}))$ lie on the real line on opposite sides of the vertical line for all values of $\delta=r^{m-2ks}>0$ as indicated in the lower part of Figure \ref{fig:motion}c).

Since the derivative of the argument of $A$ is non-zero at $t=t_k$, there is a neighbourhood $U$ of $t_k$ independent of $\delta$ such that $t=t_k$ is the only point in the neighbourhood where $\arg(\delta A)$ is 0 or $\pi$. Thus $t=t_k$ is the only point in $U$, for which the roots of $g_t+\delta A(\rme^{\rmi t})$ lie on $g_t^{-1}(\mathbb{R})$. The two roots (which are the preimage points $(g_{t_k}(h))^{-1}(-\delta A(\rme^{\rmi t_k}))$) lie on opposite sides of the vertical line at $t=t_k$ and cannot cross the vertical line while $t$ is in $U$.


Recall that a crossing only occurs when two strands have the same $\text{Re}(u)$-coordinate. Since the two preimage points remain on opposite sides of the vertical line throughout $U$, there is no crossing between the strands that are formed by the two preimage points $(g_{t}(h))^{-1}(-\delta A(\rme^{\rmi t}))$ in a neighbourhood of the original crossing for all sufficiently small $\delta>0$.

Thus all crossings at $t=t_k$, $k=1,2,\ldots,\ell'$, are resolved as in Figure \ref{fig:resolution}a), that is, there are no more crossings in the lower half of the braid.

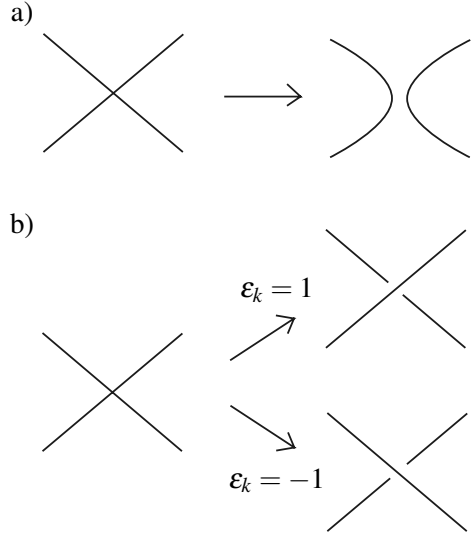
\begin{figure}
\centering
\labellist
\large
\pinlabel a) at 100 2600
\pinlabel b) at 100 1600
\pinlabel $\varepsilon_k=1$ at 1300 1300
\pinlabel $\varepsilon_k=-1$ at 1300 400
\endlabellist
\includegraphics[height=7.5cm]{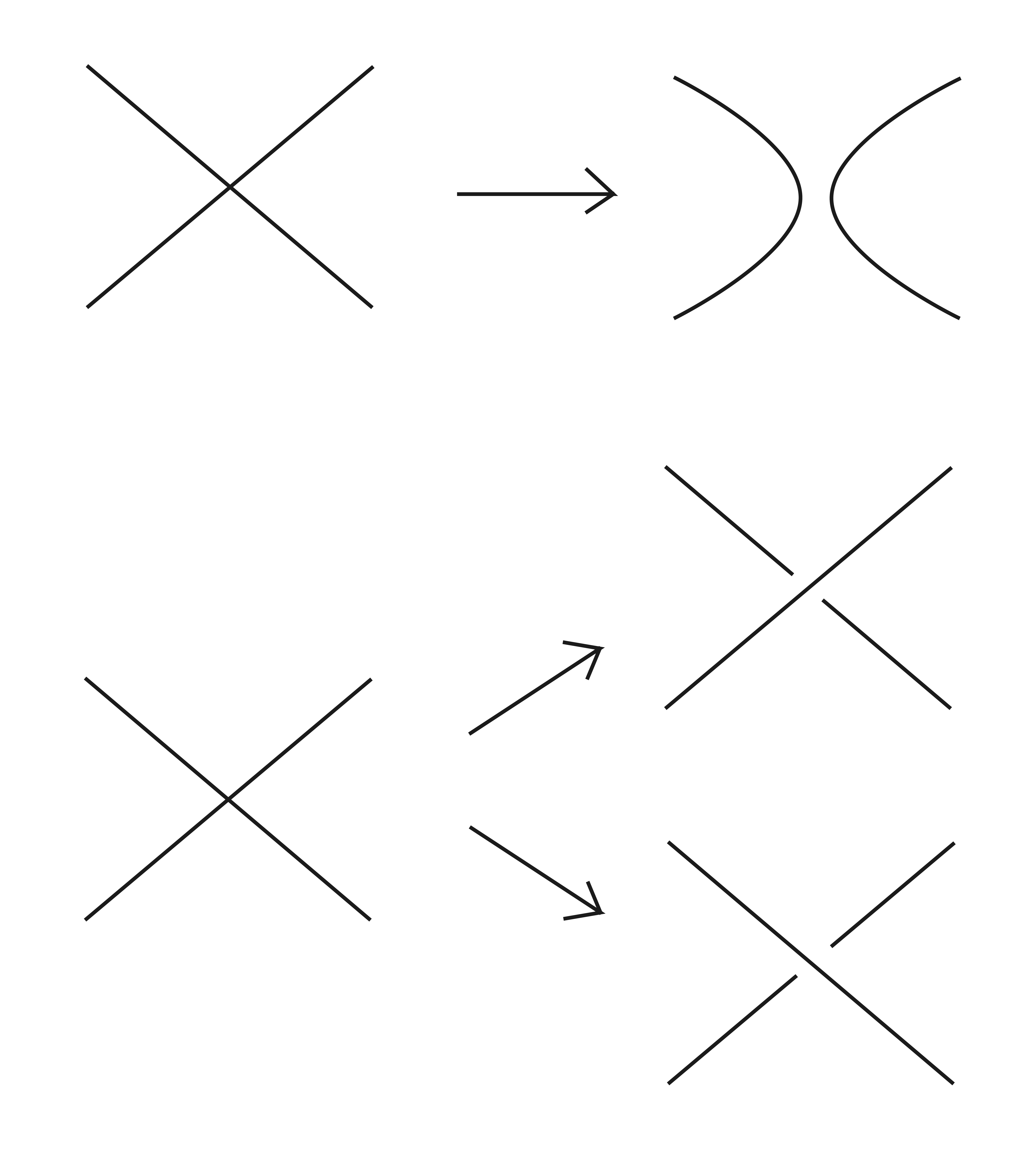}
\caption{Resolution of singular crossings. a) A singular crossing is resolved into two strands without a crossing. b) A singular crossing is resolved into a classical crossing with sign $\varepsilon_k$.\label{fig:resolution}}
\end{figure}

By symmetry $A(\rme^{\rmi t_{k+\ell'}})=A(\rme^{\rmi (t_k+\pi)})$, $k=1,2,\ldots,\ell'$, is real and has the opposite sign as the critical point associated with the crossing at $t=t_{k+\ell'}$. Therefore, the two preimage points $(g_{t}(h))^{-1}(-\delta A(\rme^{\rmi t}))$ both lie on the vertical line, one ``above'' (with positive imaginary part) the real line and one ``below'' (negative imaginary part), see the upper part of Figure \ref{fig:motion}c). Furthermore, we know that the sign of $\tfrac{\partial \arg(A)}{\partial t}$ is the sign of the desired crossing. Suppose that that sign is positive. Then the point above the real line is moving from right to left and the point below is moving from left to right, relative to the motion of the vertical line. That is, there is an $\varepsilon>0$ such that for all $t\in(t_{k+\ell'}-\varepsilon,t_{k+\ell'})$ the point above the real line is in the upper right quadrant and the point below is in the lower left quadrant, while for all $t\in(t_{k+\ell'},t_{k+\ell'}+\varepsilon)$ the point above the real line is in the upper left quadrant and the point below the real line is in the lower right quadrant.

Recall again that there is a crossing if and only if the two points have the same $\text{Re}(u)$-coordinate. This means that in $(t_{k+\ell'}-\varepsilon,t_{k+\ell'}+\varepsilon)$ there is a unique crossing, which occurs at $t=t_{k+\ell'}$. By our sign convention the sign of this crossing is positive as the point below the real line passing the point above the real line from left to right.

Likewise, if the desired sign $\varepsilon_{k}$ of the crossing is negative, then the point above the real line is moving from left to right and the point below is moving from right to left. That is, there is an $\varepsilon>0$ such that for all $t\in(t_{k+\ell'}-\varepsilon,t_{k+\ell'})$ the point above the real line is in the upper left quadrant and the point below is in the lower right quadrant, while for all $t\in(t_{k+\ell'},t_{k+\ell'}+\varepsilon)$ the point above the real line is in the upper right quadrant and the point below the real line is in the lower left quadrant. Thus there is a unique crossing at $t=t_{k+\ell'}$ and it has a negative sign.

In either case we obtain a classical crossing of the required sign as in Figure \ref{fig:resolution}b). Note that the $\varepsilon$-neighbourhood can be chosen independently of $\delta$ and thus independent of $r$, so that we have the correct crossing for all small values of $r$. Outside of the discussed neighbourhoods of $t_k$, $k=1,2,\ldots,2\ell'$, we can guarantee that there are no crossings when $r$ is sufficiently small. It follows that the zeros of $g_t(h)+r^{m-2ks}A(\rme^{\rmi t})$ form a closed braid in $\mathbb{C}\times S^1$ as $t$ varies from $0$ to $2\pi$.

Since the singular crossings in the first half of $B_{sing}^2$ at $t=t_k$, $k=1,2,\ldots,\ell'$, are all resolved into strands without crossings and the singular crossings in the second half of $B_{sing}^2$ at $t=t_k+\pi$, $k=1,2,\ldots,\ell'$, are resolved as desired, i.e., $\tau_{j_k}\mapsto \sigma_{j_k}^{\varepsilon_k}$, the braid formed by the roots of $g_t(h)+r^{m-2ks}A(\rme^{\rmi t})$ is represented by the word $\underset{k=1}{\overset{\ell'}{\prod}}\sigma_{j_k}^{\varepsilon_k}$, which by construction is braid isotopic to the braid $B$ that we used as input. Since $h$ is a diffeomorphism that preserves the fibers of the projection map onto the second factor $\mathbb{C}\times S^1\to S^1$, the roots of $g_t+r^{m-2ks}A(\rme^{\rmi t})$ form a braid that is isotopic to $B$ as a closed braid.
\end{proof}

\begin{lemma}
The constructed semiholomorphic polynomial $f$ has a weakly isolated singularity whose link is the closure of the given braid $B$.
\end{lemma}
\begin{proof}
At $v=0$ we have that $f(u,v)=u^s$, so that the origin is the only critical point with $v=0$.

We have shown that the roots of $f|_{|v|=r}$ form a braid as $t$ varies from $0$ to $2\pi$ for small values of $r>0$. In particular, all roots of $f|_{|v|=r}$ are simple, which means that $\tfrac{\partial f}{\partial u}\neq 0$ on $f^{-1}(0)\backslash \{O\}$. Thus $f$ has a weakly isolated singularity at the origin.
 
We also have that $f^{-1}(0)\cap(\mathbb{C}\times rS^1)$ is isotopic to the closed braid $B$ for all sufficiently small values of $r$. As $r$ goes to zero, the $u$-coordinates of all strands converges to zero. As in \cite{bodepoly} we can construct an explicit isotopy between the projection of $f^{-1}(0)\cap(\mathbb{C}\times rS^1)$ to $S_r^3$ and $f^{-1}(0)\cap S_r^3$ for small values of $r$, which shows that the closure of $B$ is the link of the singularity.
\end{proof}

\section{Upper bounds on the degree}\label{sec:bounds}

In this section we prove Theorem \ref{thm:bound} and Corollary \ref{cor}. The proof of the upper bound on the degree of the constructed polynomials is very similar to the one of the bound obtained in \cite{bodepoly}.
\begin{proof}
Since $F_C$ is found via trigonometric interpolation, its degree (as a trigonometric polynomial) is equal to $\left\lfloor\tfrac{x}{2}\right\rfloor$, where $x$ is the number of data points used in the interpolation and $\left\lfloor y\right\rfloor$ is the floor function that maps any real number $y$ to the largest integer less than or equal to $y$. As in \cite{bodepoly} we need $s_C\ell$ data points for the interpolation for $F_C$, so that the degree of $F_C$ is $\left\lfloor\tfrac{s_C\ell}{2}\right\rfloor$. (In \cite{bodepoly} this was erroneously stated as $\left\lfloor\tfrac{s_c\ell-1}{2}\right\rfloor$.)

We can assume that $\ell>1$, since the closure of the braid is not an unknot. As observed in Section \ref{sec:generic} the degree of $\tilde{F}_C$ is equal to the degree of $F_C$. The degree of $g$ is $\underset{C\in\mathcal{C}}{\sum}\max\{s_C,2\deg(F_C)\}\leq s\ell$. Thus $k=\left\lceil\ell/2\right\rceil$ is a choice that guarantees that $p_k$ is a polynomial, where $\left\lceil y\right\rceil$ is the smallest integer bigger than or equal to $y$.

The degree of $p_k$ is then equal to $2ks\leq s(\ell+1)$.

The trigonometric polynomial $\tilde{A}$ is found via trigonometric interpolation, where for every singular crossing of $B_{sing}$ there is one data point for the value of $\tilde{A}$ and one data point for its derivative. The degree of $\tilde{A}$ is then equal to $\ell'$, where $\ell'$ is the number of singular crossings of $B_{sing}$ \cite{nathan}. Recall that $\ell'$ could be strictly greater than $\ell$.

Singular crossings of $B_{sing}$ correspond to intersections of the curves parametrized by $\tilde{F}_C$, which correspond to the zeros of certain complex polynomials on the unit circle as in \cite{bodepoly}. It was shown in \cite{bodepoly} that the number of singular crossings that involve two strands from the same component $C$ is bounded above by $(s_C+1)s_C\ell$. (Following the mistake in \cite{bodepoly} mentioned above this bound was originally stated as $(s_C+1)(s_C\ell-1)$)

It is also shown in \cite{bodepoly} that there are at most $\ell s_Cs_{C'}$ singular crossings with one strand from the component $C$ and the other strand from a component $C'\neq C$. The total number of singular crossings and the degree of $\tilde{A}$ is bounded from above by
\begin{align}
\deg(\tilde{A})&\leq \underset{C\in\mathcal{C}}{\sum}(s_C+1)s_C\ell+\frac{1}{2}\underset{C\in\mathcal{C}}{\sum}\underset{C'\neq C}{\sum}\ell s_C s_{C'}\nonumber\\
&= \underset{C\in\mathcal{C}}{\sum}(s_C+1)s_C\ell+\frac{1}{2}\underset{C\in\mathcal{C}}{\sum}\ell s_C (s-s_C)\nonumber\\
&=\underset{C\in\mathcal{C}}{\sum}\frac{1}{2}s_C^2\ell+s\ell(1+\frac{s}{2}).
\end{align}

We need to choose $m$, which will equal the degree of $f$, to be greater than the degree of $p_k$ and at least the degree of $A$. The degree of $A$ is $2\deg(\tilde{A})+1$ and the degree of $p_k$ was at most $s(\ell+1)$. Thus 
\begin{equation}
m=\underset{C\in\mathcal{C}}{\sum}s_C^2\ell+s\ell(2+s)+1>2s\ell>s(\ell+1).
\end{equation}
is a sufficient choice.

Therefore, the degree of $f$, may be chosen to be
\begin{equation}
\deg(f)\leq\underset{C\in\mathcal{C}}{\sum}s_C^2\ell+s\ell(2+s)+1.
\end{equation}
\end{proof}

If the closure of $B$ is a knot, we have that $|\mathcal{C}|=1$ and $s_C=s$. Corollary \ref{cor} follows immediately.


\begin{thebibliography}{99}



\bibitem{ak} S. Akbulut and H. C. King.  \textit{All knots are algebraic}. Comm. Math. Helv. \textbf{56} (1981), 339--351.

\bibitem{alexander:1923lemma} J. Alexander. \textit{A lemma on a system of knotted curves}. Proc. Nat. Acad. Sci. USA {\bf 9} (1923), 93--95.

\bibitem{eder} R. N. Araújo dos Santos, B. Bode and E. L. Sanchez Quiceno. \textit{Links of singularities of inner non-degenerate mixed functions}. arXiv:2208.11655 (2022)

\bibitem{benedetti} R. Benedetti and M. Shiota. \textit{On real algebraic links on} $S^3$. Bolletino dell'Unione Mathematica Italiana, Serie 8 Volume 1B. \textbf{3} (1998), 585--609.

\bibitem{bcr:1998real} J. Bochnak, M. Coste and M.-F. Roy. \textit{Real algebraic geometry} Springer Verlag, Berlin (1998).

\bibitem{bodepoly} B. Bode and M. R. Dennis. \textit{Constructing a polynomial whose nodal set is any prescribed knot or link}. Journal of Knot Theory and its Ramifications \textbf{28}, no. 1 (2019), 1850082.

\bibitem{bode:ralg} B. Bode. \textit{Constructing links of isolated singularities of real polynomials} $\mathbb{R}^4\to\mathbb{R}^2$. Journal of Knot Theory and its Ramifications \textbf{28}, no.1 (2019), 1950009.

\bibitem{conway} J. H. Conway. \textit{An enumeration of knots and links}. In J. Leech (ed.): \textit{Computational Problems in Abstract Algebra}, Pergamon Press (1969), 329--358.

\bibitem{fukui} T. Fukui and E. Yoshinaga. \textit{The modified analytic trivialization of family of real analytic functions}. Invent. Math. \textbf{82} (1985), no. 3, 467--477.


\bibitem{king} H. C. King. \textit{Topological type of isolated critical points}. Ann. of Math. (2) \textbf{107} (1978), no. 2, 385--397.

\bibitem{looijenga} E. Looijenga. \textit{A note on polynomial isolated singularities}. Ind. Math. (Proc.) \textbf{74} (1971), 418--421.

\bibitem{nathan} A. Nathan. \textit{Trigonometric interpolation of function and derivative data}. Information and Control \textbf{28} (1975), 192--203.

\bibitem{oka} M. Oka. \textit{Non-degenerate mixed functions}. Kodai Mathematical Journal \textbf{33}, no.1 (2010), 1--62.

\bibitem{perron} B. Perron. \textit{Le n{\oe}ud ``huit'' est alg{\'e}brique r{\'e}el}. Inv. Math. \textbf{65} (1982), 441--451.

\bibitem{saeki} O. Saeki. \textit{Topological types of complex isolated hypersurface singularities}. Kodai Math. J. \textbf{12} (1989), no. 1, 23--29.
\end{thebibliography}
\end{document}